\documentclass[english,11pt]{article}
\usepackage{amsmath}
\usepackage{graphicx}
\usepackage{amssymb}
\usepackage[usenames, dvipsnames]{color}
\usepackage{tikz}
\usetikzlibrary{arrows,shapes,positioning}
\usetikzlibrary{decorations.markings}
\tikzstyle arrowstyle=[scale=1.5]
\tikzstyle{vertex}=[circle, draw, inner sep=0pt, minimum size=6pt]

\tikzstyle directed=[postaction={decorate,decoration={markings, mark=at position .5 with {\arrow[arrowstyle]{stealth}}}}]
\tikzstyle reverse directed=[postaction={decorate,decoration={markings, mark=at position .5 with {\arrowreversed[arrowstyle]{stealth};}}}]
\usepackage{}

\usepackage{amsthm}

\usepackage{enumerate}

\allowdisplaybreaks[4]

\def\0{\emptyset}

\def\p{\noindent {\bf Proof.~~}}
\def\q{\hfill\rule{1ex}{1ex}}

\newtheorem{Theorem}{Theorem}[section]

\newtheorem{Claim}{Claim}

\newtheorem{Corollary}[Theorem]{Corollary}
\newtheorem{Definition}[Theorem]{Definition}

\newtheorem{Lemma}[Theorem]{Lemma}

\newtheorem{Problem}[Theorem]{Problem}

\begin{document}

\title{Spectral radius and edge-disjoint spanning trees of graphs with edge connectivity
}

\author{Jia Wei\thanks{Department of Mathematics, South China Agricultural University, Guangzhou, 510640,
China, Email: jiawei\_math@163.com (J. Wei); liumuhou@163.com (M. Liu) (Corresponding author)},
Muhuo Liu$^*$,  
Zhifu You\thanks{School of Mathematics and Systems Science, Guangdong Polytechnic Normal University, Guangzhou 510665,
China. Email: zhifuy@hotmail.com}, 
Hong-Jian Lai\thanks{School of Mathematics and Systems Science, Guangdong Polytechnic Normal University, Guangzhou 510665,
China, and Department of Mathematics, West Virginia University,
Morgantown, WV 26506, USA. Email: hjlai2015@hotmail.com }, 
}

\date{ }
\maketitle
\begin{abstract}
Let $\tau(G)$ denote the maximum number of pairwise edge-disjoint
spanning trees contained in a graph $G$.
Zhang and Fan [J. Graph Theory, 2026] posed the problem of determining the corresponding tight spectral-radius condition for every $s$-edge-connected graph with
\[
k+1\leq s\leq2k-1,
\]
together with a characterization of the associated extremal graphs.
In this paper, we completely resolve this problem.
\end{abstract}

{\small
}

\vskip 0.35cm

\section{Introduction}
We discuss finite and simple graphs in this paper.
Undefined terms and notation in graph theory will follow those in \cite{BoMu08}.
Throughout this paper, we use $n = |V(G)|$, $e(G)$, $\delta(G)$, $\Delta(G)$, $\kappa(G)$, $\kappa'(G)$ and $\overline{G}$ to denote the order, the edge number, the  minimum degree, the maximum degree, the connectivity,
the edge connectivity and the complement of   graph $G$, respectively.
Let $G_1$ and $G_2$ be two vertex disjoint graphs.
Following \cite{BoMu08}, the {\bf disjoint union} of $G_1$ and $G_2$, denoted by $G_1+G_2$,
is the graph with vertex set $V(G_1)\cup V(G_2)$ and edge set $E(G_1)\cup E(G_2)$;
the {\bf join} of $G_1$ and $G_2$,  denoted by $G_1\vee G_2$, has vertex set $V(G_1)\cup V(G_2)$ and edge set $E(G_1)\cup E(G_2)\cup\{xy: x\in V(G_1),y\in V(G_2)\}$.
The disjoint union of $k$ copies of  graph $G$ will be referred as  $kG$.

For a graph $G$ with vertex set $V(G)=\{v_1,v_2,\ldots,v_n\}$,
the adjacency matrix of $G$ is defined to be a $(0,1)$-matrix
$A(G)=(a_{ij})$, where $a_{ij}=1$ if and only if  $v_i$ and $v_j$ are adjacent in $G$.
The largest eigenvalue of $A(G)$, denoted by $\lambda(G)$, is called the {\bf spectral radius} of $G$.


The existence of edge-disjoint spanning trees is a classical topic in
graph theory, with close connections to graph connectivity, network
reliability, routing, matroid theory, and graph decompositions. For a
graph $G$, let $\tau(G)$ denote its {\bf spanning-tree packing number}, that
is, the maximum number of pairwise edge-disjoint spanning trees
contained in $G$. The fundamental characterization of $\tau(G)$ was
established independently by Nash-Williams and Tutte.

\begin{Theorem}[Nash-Williams \cite{Nash1961}, Tutte \cite{Tutt1961}]
\label{Nash-Williams-Tutte}
Let $G$ be a connected graph and let $k$ be a positive integer. Then
$\tau(G)\geq k$ if and only if $
e_G(\pi)\geq k\bigl(|\pi|-1\bigr)$
for every partition $\pi$ of $V(G)$, where $e_G(\pi)$ denotes the
number of edges of $G$ whose end-vertices lie in distinct parts of
$\pi$.
\end{Theorem}

The Nash-Williams--Tutte theorem provides a powerful min--max
criterion for spanning-tree packings and is one of the fundamental
tools in the study of graph connectivity. For a survey on the
spanning-tree packing number, we refer the reader to Palmer
\cite{Palm2001}.

Motivated by a question of Seymour, Cioab\u{a} and Wong
\cite{CiWo2012} initiated the spectral study of the spanning-tree
packing number. Since then, some spectral conditions guaranteeing
the existence of edge-disjoint spanning trees have been established
using eigenvalues of the adjacency, Laplacian, signless Laplacian, and
distance matrices; see
\cite{COPP2022,DuWL2020,FaGL2023,FaHZ2025,Gu2016,GLLY2016,HuWD2023,LiSh2013,LHGL2014,LiHL2014,LiLT2019}
and the references therein.

Recently, Zhang and Fan \cite{ZhFa2026} obtained sharp
adjacency spectral radius conditions for $k$-edge-connected graphs
and, for $(2k-1)$-edge-connected graphs in the special case $k=2$. They
further posed the following problem.

\begin{Problem}[Zhang and Fan \cite{ZhFa2026}]\label{EDST-Problem1}
Let $k\ge 2$ and $k+1\le s\le 2k-1$. Determine the best possible lower bound on $\lambda(G)$ which guarantees $\tau(G)\ge k$ for every $s$-edge-connected graph $G$ of order $n$, and characterize all extremal graphs.
\end{Problem}

Writing $s=k+c$, where $1\leq c\leq k-1$, Gao and Wang
\cite{GaWa2026} obtained sharp spectral radius thresholds and reduced
the possible extremal graphs to certain explicitly defined structural
families. When $c=1$, they determined the unique extremal graph.
For $c\geq2$, however, which members of these families attain the
maximum spectral radius remained undetermined.

In this paper, we completely solve Problem~\ref{EDST-Problem1}.
More precisely, for   each $k\leq s\leq2k-1$, we determine the corresponding extremal  $s$-edge-connected graph  of order
$n$ with $\tau(G)\leq k-1$ having the maximum adjacency spectral
radius.

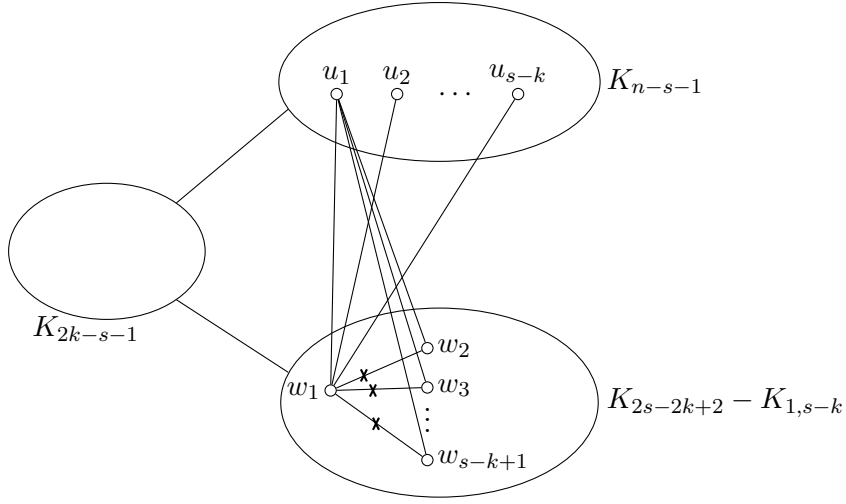
\begin{figure}[htbp]
\centering
\begin{tikzpicture}[scale=0.8]
\node[draw, ellipse,
minimum width=2.6cm,
minimum height=1.8cm]
(A) at (0.5,0) {};

\node[left] at (1.2,-1.3)
{$K_{2k-s-1}$};
\node[draw, ellipse,
minimum width=4.25cm,
minimum height=2.1cm]
(B) at (6,2.8) {};

\node[right] at (8.6,2.8)
{$K_{n-s-1}$};
\node[draw, ellipse,
minimum width=4.2cm,
minimum height=2.5cm]
(C) at (6,-2.5) {};

\node[right] at (8.6,-2.5)
{$K_{2s-2k+2}-K_{1,s-k}$};
\draw (1.65,0.8) -- (3.5,2.35);
\draw (1.65,-0.8) -- (3.5,-2.0);
\node[circle,draw,fill=white,inner sep=1.5pt]
(u1) at (4.3,2.6) {};

\node[circle,draw,fill=white,inner sep=1.5pt]
(u2) at (5.3,2.6) {};

\node at (6.3,2.6) {$\cdots$};

\node[circle,draw,fill=white,inner sep=1.5pt]
(uk) at (7.3,2.6) {};

\node[above] at (u1) {$u_1$};
\node[above] at (u2) {$u_2$};
\node[above] at (uk) {$u_{s-k}$};
\node[circle,draw,fill=white,inner sep=1.5pt]
(w1) at (4.2,-2.3) {};

\node[circle,draw,fill=white,inner sep=1.5pt]
(w2) at (5.8,-1.6) {};

\node[circle,draw,fill=white,inner sep=1.5pt]
(w3) at (5.8,-2.25) {};

\node at (5.8,-2.65) {$\vdots$};

\node[circle,draw,fill=white,inner sep=1.5pt]
(wt) at (5.8,-3.45) {};

\node[left] at (w1) {$w_1$};
\node[right] at (w2) {$w_2$};
\node[right] at (w3) {$w_3$};
\node[right] at (wt) {$w_{s-k+1}$};

\draw[thick] (4.7,-2.15) -- (4.8,-1.95);
\draw[thick] (4.7,-1.95 ) -- (4.8,-2.15);

\draw[thick] (4.85,-2.4) -- (4.95,-2.2);
\draw[thick] (4.85,-2.2) -- (4.95,-2.4);

\draw[thick] (4.9,-2.95) -- (5.0,-2.75);
\draw[thick] (4.9,-2.75) -- (5.0,-2.95);
\draw (w1) -- (w2);
\draw (w1) -- (w3);
\draw (w1) -- (wt);
\draw (w1) -- (u1);
\draw (w1) -- (u2);
\draw (w1) -- (uk);

\draw (u1) -- (w2);
\draw (u1) -- (w3);
\draw (u1) -- (wt);
\end{tikzpicture}
\caption{The graph $H_1(s,k)$.~~~}\label{fig1}
\end{figure}

Let $H_1(s,k)$ (see Figure \ref{fig1}) be the graph obtained from
$K_{2k-s-1}\vee\Bigl(K_{n-s-1}+\bigl(K_{2s-2k+2}-E(K_{1,s-k})\bigr)\Bigr)$
by adding $2(s-k)$ edges between $K_{n-s-1}$ and $K_{2s-2k+2}-E(K_{1,s-k})$
such that each vertex of $K_{2s-2k+2}- E(K_{1,s-k})$ has degree $s$, and all leaves in $K_{1,s-k}$ are adjacent to a common vertex of $K_{n-s-1}$.
Then  $H_1(s,k)$ is an $s$-edge-connected graph (see Lemma \ref{EDST-connect-le1} bellow).

\begin{Theorem}\label{spectral-spanning-tree-thm1}
Let $G$ be an $s$-edge-connected graph of order $n$, where $n\ge 3s^2$ and $2\le k\le s\le 2k-2$. If
$\lambda (G) \ge \lambda(H_1(s,k)),$
then $\tau(G)\ge k$ unless $G\cong H_1(s,k)$.
\end{Theorem}

Let $H_2$ be the graph obtained from two copies of $K_{1,k}$ by identifying one leaf from each copy. Denote by $H_2(k)$ (see Figure \ref{fig2})  the graph obtained from  $K_{n-2k-1}+\bigl(K_{2k+1}-E(H_2)\bigr)$ by adding $2k-1$ edges between $K_{n-2k-1}$ and $K_{2k+1}-E(H_2)$ such that every vertex of $K_{2k+1}-E(H_2)$ has degree exactly $2k-1$. Moreover, the two vertices of degree $k$ in $H_2$ are adjacent to the same $k-1$ vertices of $K_{n-2k-1}$, and one of these $k-1$ vertices is also adjacent to the common leaf vertex of $H_2$. Then $H_2(k)$ is a  $(2k-1)$-edge-connected graph. (see Lemma \ref{EDST-thm2-le1} bellow).

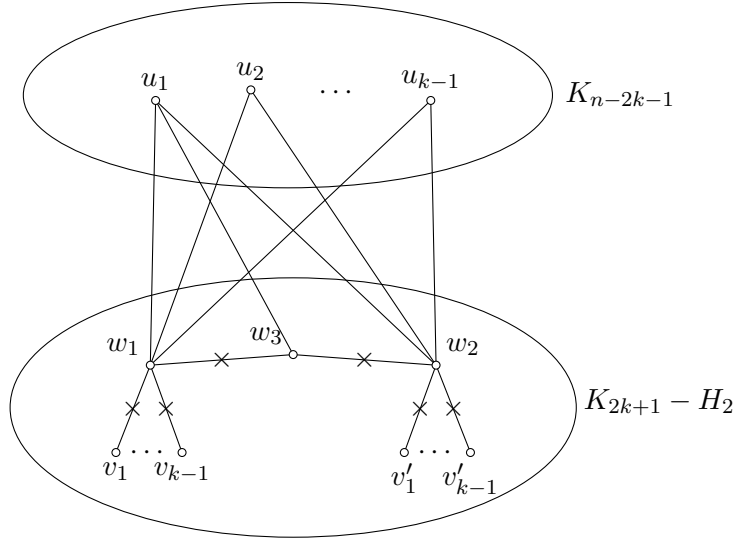
\begin{figure}[htbp]
\centering
\begin{tikzpicture}[scale=0.7]

\coordinate (u1) at (-2.6,5.8);
\coordinate (u2) at (-0.8,6.0);
\coordinate (uk) at ( 2.6,5.8);

\coordinate (w1) at (-2.7,0.8);
\coordinate (w3) at ( 0.0,1.0);
\coordinate (w2) at ( 2.7,0.8);

\coordinate (v1)  at (-3.35,-0.85);
\coordinate (vk)  at (-2.10,-0.85);
\coordinate (vp1) at ( 2.10,-0.85);
\coordinate (vpk) at ( 3.35,-0.85);

\draw (-0.1,5.9) ellipse (5.0cm and 1.75cm);
\draw ( 0.0,0.0) ellipse (5.35cm and 2.45cm);
\node[right] at (4.95,5.90) {$K_{n-2k-1}$};
\node[right] at (5.30,0.10) {$K_{2k+1}-H_2$};

\draw (u1)--(w1);
\draw (u1)--(w2);
\draw (u1)--(w3);
\draw (u2)--(w1);
\draw (u2)--(w2);
\draw (uk)--(w1);
\draw (uk)--(w2);

\draw (w1)--(w3)--(w2);
\draw (w1)--(v1);
\draw (w1)--(vk);
\draw (w2)--(vp1);
\draw (w2)--(vpk);

\node[inner sep=2pt] at (-1.35,0.90) {$\times$};
\node[inner sep=2pt] at ( 1.35,0.90) {$\times$};
\node[inner sep=2pt] at (-3.03,-0.02) {$\times$};
\node[inner sep=2pt] at (-2.40,-0.02) {$\times$};
\node[inner sep=2pt] at ( 2.40,-0.02) {$\times$};
\node[inner sep=2pt] at ( 3.03,-0.02) {$\times$};

\draw[fill=white] (u1) circle (0.075cm);
\draw[fill=white] (u2) circle (0.075cm);
\draw[fill=white] (uk) circle (0.075cm);
\draw[fill=white] (w1) circle (0.075cm);
\draw[fill=white] (w3) circle (0.075cm);
\draw[fill=white] (w2) circle (0.075cm);
\draw[fill=white] (v1) circle (0.075cm);
\draw[fill=white] (vk) circle (0.075cm);
\draw[fill=white] (vp1) circle (0.075cm);
\draw[fill=white] (vpk) circle (0.075cm);

\node[above] at (u1) {$u_1$};
\node[above] at (u2) {$u_2$};
\node[above] at (uk) {$u_{k-1}$};
\node at (0.85,5.95) {$\cdots$};

\node[above left]  at (w1) {$w_1$};
\node[above left]  at (w3) {$w_3$};
\node[above right] at (w2) {$w_2$};

\node[below] at (v1)  {$v_1$};
\node[below] at (vk)  {$v_{k-1}$};
\node[below] at (vp1) {$v'_1$};
\node[below] at (vpk) {$v'_{k-1}$};
\node at (-2.72,-0.85) {$\cdots$};
\node at ( 2.72,-0.85) {$\cdots$};

\end{tikzpicture}
\caption{The graph $H_2(k)$.~~~~~~~~~}\label{fig2}
\end{figure}

\begin{Theorem}\label{spectral-spanning-tree-thm2}
Let $G$ be a $(2k-1)$-edge-connected graph of order $n\ge 8k^2-2k$, where $ k\ge 2$. If $
\lambda (G) \ge \lambda(H_2(k)),$
 then $\tau(G)\ge k$ unless $G\cong H_2(k)$.
\end{Theorem}

The rest of this paper is organized as follows. In Section~2, we present several preliminary results. In Section~3, we prove our main theorems.

\section{Preliminaries}
The purpose of this section is to establish some auxiliary results that will be used in the proofs of our main theorems. We begin by recalling some known results.

\begin{Lemma}  {\em(Hong et al. \cite{HoSF2001}, and Nikiforov \cite{lesN2002})} \label{le3-1}
Let $G$ be a graph with  $n$ vertices, $e(G)$ edges  and    minimum degree $\delta(G)$. If $\delta(G)\ge r$, then
\[
\lambda(G)\leq \frac{r-1+\sqrt{{(r+1)}^2+4(2e(G)-nr)}}{2}.
\]
\end{Lemma}

\begin{Lemma} {\em (Wei et al. \cite{WYSL22})}\label{le0-1}
Let $r, a_1, a_2,\ldots, a_r$ be integers with $r \ge 2$
and $0 \le a_1\le a_2\le\dots\le a_r$, and let $f(x_1,x_2,\dots,x_r)=\sum_{i=1}^r \binom{x_i}{2} $
be an integral function with $\sum_{i=1}^r x_i=n$. If $a_i\le x_i$ holds  for any $i\in\{1,2,\ldots,r\}$, then
\begin{align*}
f(x_1,x_2,\dots,x_r)\le \binom{n-\sum_{i=1}^{r-1} a_i}{2}+\sum_{i=1}^{r-1} \binom{a_i}{2},
\end{align*}
where  the equality holds if and only if $x_i=a_i$ for any  $i\in\{1,2,\ldots, r-1\}$ and $x_r=n-\sum_{i=1}^r a_i$.
\end{Lemma}

\begin{Lemma} {\em (Fan et al. \cite{FaGL2023})}\label{le0-2}
Let $a$ and $b$ be two positive integers. If $a\ge b$,
then $\binom{a}{2}+\binom{b}{2}<\binom{a+1}{2}+\binom{b-1}{2}$.
\end{Lemma}

From Lemmas \ref{le0-1} and \ref{le0-2}, it easily follows that
\begin{Corollary}  \label{21c}
Let $r, a_1, a_2,\ldots, a_r$ be  integers with $r \ge 2$ and $0 \le a_1\le a_2\le\dots\le a_r$. Let $f(x_1,x_2,\dots,x_r)=\sum_{i=1}^r \binom{x_i}{2} $
be an integral function with $\sum_{i=1}^r x_i=n$. If $a_i\le x_i$ holds  for any $i\in\{1,2,\ldots,r\}$, then
\begin{align*}
f(x_1,x_2,\dots,x_r)\le \binom{n-\sum_{i=1}^{r-1} a_i}{2}+\sum_{i=1}^{r-1} \binom{a_i}{2}\le \binom{n-r+1}{2}.\end{align*}
\end{Corollary}

Let $u, v$ be two distinct vertices  of graph $G$.  Then, $N_G(v)$ and $N_G[v]$ denotes the neighbor set and close neighbor set of $v$, respectively.
Define $G(u,v)$ to be the graph obtained from $G$ by replacing all edges $vw$ by $uw$ for each $w\in N_G(v)\backslash N_G[u]$.
This operation that converts $G$ to $G(u,v)$ is called the {\bf Kelmans transformation} \cite{Kelm81}.   In \cite{Csik09}, Csikv\'{a}ri showed that $
\lambda(G)\le \lambda(G')$ if $G'$ is the graph obtained from $G$ by some Kelmans transformation. Recently, Liu et al. \cite{Mu2019} improved this result slightly.

\begin{Lemma}  \label{le2.1}
Let $G$ be a connected  graph and $G'$ be the graph obtained from $G$ by some Kelmans transformation.
 Then each of the following holds:
\\
$(i)$\,{\em (Csikv\'{a}ri \cite{Csik09})} $\lambda(G)\le \lambda(G')$;
$(ii)$\,{\em(Liu et al. \cite{Mu2019})} if $G\ne G'$, then
$\lambda(G)<\lambda(G').$
\end{Lemma}
As $A(G)$ is an irreducible nonnegative matrix for a connected graph $G$, by the Perron-Frobenius Theorem, $A(G)$ has a positive eigenvector $\mathbf x $ corresponding to  $\lambda(G)$, which is called the {\bf Perron vector} of $A(G)$.
For any $v\in V(G)$, let $x_v$ be the coordinate in $\mathbf x$   corresponding to vertex $v$.

\begin{Lemma}{\em (Wu et al. \cite{WuXY05})} \label{le2.2}
 Let $G$ be a connected graph, and  $u,v$ be two vertices of $G$.
Suppose that $v_1,v_2,\ldots,v_s \in N_G(v)\setminus N_G[u]$
with $s\ge 1$, and let $G^{*}$ be the graph obtained from $G$ by deleting the edges
$vv_i$ and adding the edges $uv_i$ for $1\le i\le s$.
Let $\mathbf x $ be the Perron vector of $A(G)$.
If $x_u\ge x_v$, then $\lambda(G)<\lambda(G^{*})$.
\end{Lemma}

We next establish several lemmas that will be used in the proof of Theorem~\ref{spectral-spanning-tree-thm1}.

\begin{Definition}\label{EDST-connect-def1}
Let $G_1=K_{n-2s+2k-2}$ and $G_2=(2s-2k+2)K_1$  be two vertex-disjoint graphs, where $k$, $s$ and $n$ are positive integers satisfying that $n\ge 3s-2k+3$ and $2\le k\le s\le 2k-2$. Let
$\mathcal{H}_1(s,k)=\{G: G_1+G_2\subseteq G \text{~and~} \delta(G)\ge s\}.$
\end{Definition}

\begin{Lemma}\label{EDST-connect-le1}
Every graph in $\mathcal{H}_1(s,k)$ is $s$-edge-connected.
\end{Lemma}
\p By contradiction, assume  that there exists $G\in \mathcal{H}_1(s,k)$ such that $\kappa'(G)\le s-1$.
Choose a nonempty proper subset $V_1\subseteq V(G)$ such that
$e(V_1,V_2) = \kappa'(G)\le s-1$, where $V_2=V(G)\setminus V_1$.
  Then we have the following result.

\begin{Claim}\label{EDST-calim1}
Either $V_1\cap V(G_1)=\emptyset$ or $V_2\cap V(G_1)=\emptyset$
\end{Claim}
\p
On the contrary, suppose that both $V_1\cap V(G_1)\neq \emptyset$ and $V_2\cap V(G_1)\neq \emptyset$. Let $|V_1\cap V(G_1)|=t_1>0$ and $|V_2\cap V(G_1)|=t_2>0$.
Then we have $t_1+t_2=n-2s+2k-2$ and $e(V_1,V_2)\ge e(V_1\cap V(G_1),V_2\cap V(G_1))=t_1t_2$. Since $t_1>0$, $t_2>0$ and $t_1+t_2=n+2k-2s-2$, it follows that
$1\le t_1 \le n-2s+2k-3$.  By algebraic manipulations, we have  $t_1t_2=t_1(n+2k-2s-2-t_1)$, which has a minimum value $n-2s+2k-3$ with $1\le t_1 \le n-2s+2k-3$.  Since $n\ge 3s-2k+3$,  it follows that
$e(V_1,V_2)\ge t_1t_2\ge n+2k-2s-3 \ge s$, a contradiction. This proves Claim \ref{EDST-calim1}.
\q

By Claim \ref{EDST-calim1}, we have either $V(G_1)\subseteq V_2$ or $V(G_1)\subseteq V_1$. Without loss of generality, we assume that  $V(G_1)\subseteq V_2$.  As $V(G_1)\cup V(G_2)=V_1\cup V_2=V(G)$,  we have $V_1\subseteq V(G_2)$  and $V(G_2)\cap V_1\neq \emptyset$. Let $|V(G_2)\cap V_1|=t_3$.
As $V(G_2)\cap V_1\neq \emptyset$ and $s\le 2k-2$,
we have $1\le t_3 = |V(G_2)\cap V_1|\le |V(G_2)| = 2s-2k+2 \le s$.
As $\delta(G)\ge s$, we have $e(V_1,V_2)= e(V_1\cap V(G_2),V_2)\ge \delta(G)\times |V(G_2)\cap V_1| - 2\binom{t_3}{2} \ge st_3- t_3^2+t_3$, which has a minimum value $s$, a contradiction.
Hence, every graph in $\mathcal{H}_1(s,k)$ is $s$-edge-connected. This proves  Lemma \ref{EDST-connect-le1}.
\q

\begin{Lemma}\label{EDST-connect-le2}
Let $G\in \mathcal{H}_1(s,k) $ and  $n\ge 3s^2$. If
\begin{align*}
 e(G)=\binom{n-2s+2k-2}{2}+2k(s-k+1)-1,
\end{align*}
then $n-2s+2k-3<\lambda(G)<n-2s+2k-2$.
\end{Lemma}
\p
By Definition of $\mathcal{H}_1(s,k)$, we have $G_1=K_{n-2s+2k-2}\subsetneqq G$, and so $ \lambda(G)>\lambda(G_1)=n-2s+2k-3$.

We next prove that $\lambda(G)<n-2s+2k-2$. On the contrary, suppose that $\lambda(G)\ge n-2s+2k-2$.
By Lemma \ref{le3-1}, we have
\[
n-2s+2k-2\le \lambda(G)\leq \frac{s-1+\sqrt{{(s+1)}^2+4(2e(G)-ns)}}{2},
\]
which implies that
\[
e(G)\ge
\frac{(2n-5s+4k-3)^2-(s+1)^2+4ns}{8}.
\]
As $n\ge 3s^2$ and $2\le k\le s\le 2k-2$, we have
\begin{align*}
e(G)\ge & \frac{(2n-5s+4k-3)^2-(s+1)^2+4ns}{8}
\\
=& \binom{n-2s+2k-2}{2}+2k(s-k+1)-1 +n+2k^2-3sk+s^2-\frac{3}{2}s-1
\\
\ge & \binom{n-2s+2k-2}{2}+2k(s-k+1)-1
+(2k^2+s^2-\frac{3}{2}s-1 )
\\
> & \binom{n-2s+2k-2}{2}+2k(s-k+1)-1
=e(G),
\end{align*}
a contradiction. Hence, $n-2s+2k-3<\lambda(G)<n-2s+2k-2$.
This proves Lemma \ref{EDST-connect-le2}.\q

Finally, we establish the following lemma, which will be used in the proof of Theorem \ref{spectral-spanning-tree-thm2}.

\begin{Definition}\label{EDST-connect-def2}
Let $G_1=K_{n-2k-1}$ and $G_2=(2k+1)K_1$ be two vertex-disjoint graphs, where $n\ge 4k+1$. Define  $\mathcal{H}_2(k)$ to be the family of all graphs $G$ obtained from the disjoint union of $G_1$ and $G_2$ by adding edges within $V(G_2)$ and between $V(G_1)$ and $V(G_2)$, such that $\Delta(\overline{G}[V(G_2)])\le k$,
$e(V(G_1),V(G_2))=2k-1$ in $G$ and $d_G(v)=2k-1$ for any $v\in V(G_2)$.
\end{Definition}

For vertex $v\in V(G)$ and a subset $S\subseteq V(G)$, whether $v\in S$ or not, denote by $N_S(v)=N_G(v)\cap S$ and $d_S(v)=|N_S(v)|$. In particular, when $S=V(H)$ for some induced subgraph $H$ of $G$, we just simplify write  $d_{V(H)}(v)$ as $d_H(v)$.
\begin{Lemma}\label{EDST-thm2-le1}
Every graph in $\mathcal{H}_2(k)$ is $(2k-1)$-edge-connected.
\end{Lemma}
\p
By contradiction, assume  that there exists $G\in \mathcal{H}_2(k)$ such that $\kappa'(G)\le 2k-2$.
Choose a nonempty proper subset $V_1\subseteq V(G)$ such that
$e(V_1,V_2) = \kappa'(G)\le 2k-2$, where $V_2=V(G)\setminus V_1$.
  Then we have the following result.

\begin{Claim}\label{EDST-th2-le-calim2}
Either $V_1\cap V(G_1)=\emptyset$ or $V_2\cap V(G_1)=\emptyset$
\end{Claim}
\p
On the contrary, suppose that both $V_1\cap V(G_1)\neq \emptyset$ and $V_2\cap V(G_1)\neq \emptyset$. Let $|V_1\cap V(G_1)|=t_1>0$ and $|V_2\cap V(G_1)|=t_2>0$.
Then we have $t_1+t_2=n-2k-1$ and $e(V_1,V_2)\ge e(V_1\cap V(G_1),V_2\cap V(G_1))=t_1t_2$. Since $t_1>0$, $t_2>0$ and $t_1+t_2=n-2k-1$, it follows that
$1\le t_1 \le n-2k-2$.  By algebraic manipulations, we have  $t_1t_2=t_1(n-2k-1-t_1)$, which has a minimum value $n-2k-2$ with $1\le t_1 \le n-2k-2$.  Since $n\ge 4k+1$,  it follows that
$e(V_1,V_2)\ge t_1t_2\ge n-2k-2 \ge 2k-1$, a contradiction. This proves Claim \ref{EDST-th2-le-calim2}.\q

By Claim \ref{EDST-th2-le-calim2}, we have either $V(G_1)\subseteq V_2$ or $V(G_1)\subseteq V_1$. Without loss of generality, we assume that  $V(G_1)\subseteq V_2$.  As $V(G_1)\cup V(G_2)=V_1\cup V_2=V(G)$,  we have $V_1\subseteq V(G_2)$  and $V(G_2)\cap V_1\neq \emptyset$. Let $|V(G_2)\cap V_1|=t_3$.
Then $1\le t_3 = |V(G_2)\cap V_1|\le |V(G_2)| = 2k+1$.
Since $e(V_1,V_2) = \kappa'(G)\le 2k-2$ and  $e(V(G_1),V(G_2))=2k-1$, we get  $V(G_2)\cap V_1 \neq V(G_2)$. And so $1\le t_3 = |V(G_2)\cap V_1|\le 2k$. We first assume that $1\le t_3\le 2k-1$.
As $\delta(G)\ge 2k-1$, we have $e(V_1,V_2)= e(V_1\cap V(G_2),V_2)\ge \delta(G)\times |V(G_2)\cap V_1| - 2\binom{t_3}{2} \ge  2kt_3- t_3^2$, which has a minimum value $2k-1$, a contradiction.

Thus $ t_3=2k$.  And so $|V(G_2)\setminus V_1|=1$. Assume that $V(G_2)\setminus V_1=\{w\}$.  Then $d_G(w)=d_{G_1}(w)+d_{G_2}(w)=2k-1$.
Since $\Delta(\overline{G}[V(G_2)])\le k$, it follow that $d_{G_2}(w)\ge |V(G_2)|-1-k=k$.
Therefore, $e(V_1,V_2)=e(V(G_2)\setminus \{w\},V(G_1)\cup \{w\})= e(V(G_2),V(G_1))+d_{G_2}(w)-d_{G_1}(w)=
e(V(G_2),V(G_1))+2d_{G_2}(w)-d_G(w)=2d_{G_2}(w)\ge 2k$, a contradiction.

Hence, every graph in $\mathcal{H}_2(k)$ is $(2k-1)$-edge-connected. This proves  Lemma \ref{EDST-thm2-le1}.
\q

\section{Proofs of Theorems \ref{spectral-spanning-tree-thm1} and \ref{spectral-spanning-tree-thm2}}
 Let $G$ be the graph with the maximum spectral radius  among all $s$-edge-connected graphs satisfying $\tau(G)\le k-1$. Let $\mathbf x$ be the Perron vector of $A(G)$, and  $x_v$ be the coordinate in $\mathbf x$  corresponding to  vertex $v\in V(G)$.
 As $\tau(G)\le k-1$, by Theorem \ref{Nash-Williams-Tutte} and the choice of $G$, there exists a partition  $\pi$ of $V(G)$ such that
\begin{align}\label{spectral-spanning-tree-eq1}
 e_{G}(\pi) = k(|\pi|-1)-1.
\end{align}
Let $\mathcal{P}$ be the set of all partitions of $V(G)$  satisfying
(\ref{spectral-spanning-tree-eq1}). Choose a partition $\pi=({V_1,V_2,\dots,V_p})\in \mathcal{P}$ such that the number of singleton parts in $\pi$  is as large as possible. Without loss of generality, we assume $V_1,V_2,\dots,V_t$ are the singleton parts and $2\le |V_{t+1}|\le |V_{t+2}|\le \cdots \le|V_p|$  in $\pi$,  where $t$ is a nonnegative integer.

Since $G$ is $s$-edge-connected, we have $e(V_i,V(G)\setminus V_i)\ge s$ for any $1\le i\le p$. Then
\[
\frac{sp}{2}\le e_G(\pi)=\frac{1}{2}\sum_{i=1}^p e(V_i,V(G)\setminus V_i)= k(p-1)-1,
\]
which implies
\begin{align}\label{spectral-spanning-tree-eq8}
  p(2k-s) \ge 2k+2.
\end{align}

\begin{Claim}\label{spectral-spanning-tree-claim1}
Each of the following holds:
\\
(i)   $p=|\pi|\ge   3$ for $k\le s\le 2k-2$ and $p=|\pi|\ge  2(k+1)$ for $s=2k-1$.
\\
(ii) for any $i\in \{t+1,t+2,\ldots,p\}$, $|V_i|\ge k+2$.
\end{Claim}
\p
By \eqref{spectral-spanning-tree-eq8}, we get $p\ge \frac{2(k+1)}{2k-s}$. Thus $p=|\pi|\ge   3$ for $k\le s\le 2k-2$ and $p=|\pi|\ge  2(k+1)$ for $s=2k-1$. This prove (i).

We prove (ii) by contradiction, suppose that there exists $i\in \{t+1,t+2,\ldots,p\}$ such that $|V_i|\le k+1$. Then $2\le |V_i|\le k+1$. Assume that $w_1\in V_i$.  Let $\pi_1=(V_1,\ldots,V_{i-1},V_{i}\setminus \{w_1\},V_{i+1},\ldots,V_{p},\{w_1\})$. Then $\pi_1$ is a partition of $V(G)$ with $|\pi_1|=p+1$, and
\begin{align*}
 e_{G}(\pi_1) = e_{G}(\pi)+d_{V_{i}}(w_1)\le k(p-1)-1+k=kp-1=k(|\pi_1|-1)-1.
\end{align*}
Thus $\pi_1$ is a partition of $V(G)$ satisfying
(\ref{spectral-spanning-tree-eq1}).  Furthermore, $V_1,V_2,\dots,V_t,\{x_1\}$ are the singleton parts in $\pi_1$, and so it contradicts with the choice of $\pi$. This proves (ii).  So Claim \ref{spectral-spanning-tree-claim1} holds.
\q

As $e(G)\le \sum_{i=t+1}^p\binom{|V_i|}{2}+e_G(\pi)$, by Claim \ref{spectral-spanning-tree-claim1}(ii) and Corollary \ref{21c},
we have
\begin{align}\label{spectral-spanning-tree-eq-add1}
e(G)&\le \sum_{i=t+1}^p\binom{|V_i|}{2}+e_G(\pi) \nonumber
\\
&\le \binom{n-(p-t-1)(k+2)-t}{2}+(p-t-1)\binom{k+2}{2}+k(p-1)-1
 \nonumber \\
&\le \binom{n-p+1}{2}+k(p-1)-1.
\end{align}

\subsection{Proof of Theorem \ref{spectral-spanning-tree-thm1} }

As $K_{n-2s+2k-2}\subsetneqq H_1(s,k)$, we have $\lambda(G)>n-2s+2k-3$.
Combining this with Lemma \ref{le3-1}, we get
\[
n-2s+2k-3<\lambda(G)\le \frac{s-1+\sqrt{(s+1)^2+4(2e(G)-ns)}}{2},
\]
which implies
\begin{align}\label{spectral-spanning-tree-eq0}
e(G) > \frac{(2n-5s+4k-5)^2-(s+1)^2+4ns}{8}.
\end{align}


\begin{Claim}\label{spectral-spanning-tree-claim2}
 $p=|\pi| \le 2s-2k+3$.
\end{Claim}
\p
Suppose, to the contrary, that $p \ge 2s-2k+4$.
By (\ref{spectral-spanning-tree-eq-add1}) and (\ref{spectral-spanning-tree-eq0}), we have
\[
\binom{n-p+1}{2}+k(p-1)-1
>
\frac{(2n-5s+4k-5)^2-(s+1)^2+4ns}{8},
\]
which is equivalent to
\begin{align}\label{EDSTconnectivity-eq1}
f(p,s)>0,
\end{align}
where
$f(p,s)=\frac{1}{2}p^2 - np + (k-\frac{1}{2})p + 2ns - 2nk + 3n
- 3s^2 - 2k^2 + 5sk - 6s + 4k - 4.$

Since $f(p,s)$ is a convex quadratic function in $p$, it follows that for $p\in [2s-2k+4, n]$, the maximum of $f(p,s)$ is attained at an endpoint. Therefore,
\begin{align}\label{sEDSTconnectivity-eq2}
f(p,s)\le \max\{f(2s-2k+4,s),\,f(n,s)\}.
\end{align}
As
\begin{align*}
f(2s-2k+4,s)=&-s^2+3sk-2k^2+s+k-n+2,
\\
f(n,s)=& -\frac{1}{2}n^2 + 2ns - nk + \frac{5}{2}n - 3s^2 - 2k^2 + 5sk - 6s +4k - 4,
\end{align*}
and $n\ge 3s^2$, it follows from the above ranges of $s$, where $k\le s\le 2k-2$, that
\begin{align*}
f(2s-2k+4,s)&=-\left(s-\frac{3k+1}{2}\right)^2
  +\frac{k^2+10k-4n+9}{4} \le
\frac{k^2 + 10k - 4n + 9}{4}\le 0 \text{~and}
   \\
f(n,s)&\le f(n,2k-2)=-\frac{1}{2}n^2 + 3nk - \frac{3}{2}n - 4k^2 + 6k - 4\le 0.
\end{align*}
Then by (\ref{sEDSTconnectivity-eq2}), we have  $f(p,s)\le \max\{f(2s-2k+4,s),\,f(n,s)\}\le 0$, contrary to (\ref{EDSTconnectivity-eq1}). So Claim \ref{spectral-spanning-tree-claim2} holds.
\q

\begin{Claim}\label{spectral-spanning-tree-claim3}
 $t=p-1$.
\end{Claim}
\p
On the contrary, suppose that $t\neq p-1$. Then $t=p=n$ or $0\le t\le p-2$.
As $n\ge 3s^2$ and $p\le 2s-2k+3$ in Claim \ref{spectral-spanning-tree-claim2}, we have $t=p=n$ is impossible. Thus we have $0\le t\le p-2$.  We next discuss the following two cases.

\noindent
\textbf{Case 1.} $0\le t\le p-3$. By Claim \ref{spectral-spanning-tree-claim1}(ii)  and Corollary \ref{21c}, we have
\begin{align*}
e(G)\le& \binom{n-(p-t-1)(k+2)-t}{2}+(p-t-1)\binom{k+2}{2}+k(p-1)-1,
\\
\le& \binom{n-2k-p-1}{2}+2\binom{k+2}{2}+k(p-1)-1.
\end{align*}
Combining this with (\ref{spectral-spanning-tree-eq0}), we have
\begin{align*}
\frac{(2n-5s+4k-5)^2-(s+1)^2+4ns}{8}< \binom{n-2k-p-1}{2}+2\binom{k+2}{2}+k(p-1)-1,
\end{align*}
it follows that
\begin{align}\label{spectral-spanning-tree-eq4}
h(p,s) >0,
\end{align}
where $h(p,s)=\frac{1}{2}p^2+(3k-n+\frac{3}{2})p+2ns-4nk+n-3s^2+5sk+k^2-6s+10k-1$.

Since $h(p,s)$ is a convex  quadratic function in $p$ with positive leading coefficient, it follows that for $p\in [3,\, 2s-2k+3]$, the maximum of $h(p,s)$ is attained at an endpoint. Therefore,
\begin{align}\label{spectral-spanning-tree-eq5-0}
h(p,s)\le\max\bigl\{h(3,s),h(2s-2k+3,s)\bigr\}.
\end{align}
As
\begin{align*}
h(3,s)
= & 2ns-4nk-2n-3s^2+5sk+k^2-6s+19k+8,
\\
h(2s-2k+3,s)
= & -2nk-2n-s^2+7sk-3k^2+3s+10k+8,
\end{align*}
 $2\le k\le s\le 2k-2$ and $n\ge 3s^2$, we have $h(3,s)$ is increasing of $s$  on $[k,2k-2]$ and
\begin{align*}
h(2s-2k+3,s)-h(3,s)=(s-k)(2s+4k+9-2n)\le 0.
\end{align*}
Combining these with \eqref{spectral-spanning-tree-eq5-0}, we have
\begin{align*}
h(p,s)\le h(3,s)\le h(3,2k-2) =-6n-k^2+21k+8<0,
\end{align*}
contrary to (\ref{spectral-spanning-tree-eq4}).

\noindent
\textbf{Case 2.} $t=p-2$. Then by \eqref{spectral-spanning-tree-eq1} and $G$ is $s$-edge-connected, we have $k(p-1)-1=e_G(\pi)\ge (p-2)s-e(G[\cup_{i=1}^{p-2}V_i])$, and so $e(G[\cup_{i=1}^{p-2}V_i])\ge (p-2)(s-k)+1-k$.  Therefore, $(p-2)(s-k)+1-k\le e(G[\cup_{i=1}^{p-2}V_i])\le \binom{p-2}{2}$, it follows that
\begin{align}\label{spectral-spanning-tree-eq9-3}
(p-3-2s+2k)(p-2)\ge 2(1-k).
\end{align}

We next prove that $\lambda(G)<n-k-p$. Suppose, to the contrary, that
$\lambda(G)\ge n-k-p$. By $G$ is $s$-edge-connected and Lemma \ref{le3-1},
\[
n-k-p\le \lambda(G)
\le
\frac{s-1+\sqrt{(s+1)^2+4(2e(G)-ns)}}{2},
\]
which implies that
\begin{align}\label{spectral-spanning-tree-eq0-3}
e(G)\ge \frac{(2n-2k-2p-s+1)^2-(s+1)^2+4ns}{8}.
\end{align}
On the other hand, by Claim
\ref{spectral-spanning-tree-claim1}(ii) and Lemma \ref{le0-1},
\[
e(G)\le
\binom{n-k-p}{2}
+\binom{k+2}{2}
+k(p-1)-1.
\]
Combining this  with \eqref{spectral-spanning-tree-eq0-3}, we obtain
\begin{align}\label{spectral-spanning-tree-eq4-3}
&\binom{n-k-p}{2}+\binom{k+2}{2}+k(p-1)-1
-\frac{1}{8}[(2n-2k-2p-s+1)^2-(s+1)^2+4ns] \nonumber
\\
=&\frac{1}{2}\bigl(p(2k+2-s)+k(k+3-s)+s-2n\bigr)\ge0.
\end{align}
As $2\le k\le s\le 2k-2$, $n\ge 3s^2$ and  $p\le 2s-2k+3$, we have
 \begin{align*}
& p(2k+2-s)+k(k+3-s)+s-2n
 \\
 \le & (2s-2k+3)(2k+2-s)+k(k+3-s)+s-6s^2
\\
=&-6k^2+7k+6+(s-k)(2-11k-8s+8k)<0,
\end{align*}
contrary to \eqref{spectral-spanning-tree-eq4-3}.
Thus  $\lambda(G)<n-k-p$.

As $n-2s+2k-3<\lambda(G)<n-k-p$, we have
\begin{align}\label{spec-span-tree-case3-1}
p\le 2s-3k+2.
\end{align}
Combining this with \eqref{spectral-spanning-tree-eq9-3} and $p\ge 3$,  we get $2(1-k)\le (p-3-2s+2k)(p-2)\le  -(1+k)(p-2)$.
Then $2(k-1)\ge (1+k)(p-2)$, and  so $p=3$.  Combining this with \eqref{spec-span-tree-case3-1} and \eqref{spectral-spanning-tree-eq8},
we have $ \frac{3k+1}{2}\le s\le \frac{4k-2}{3}$, a contradiction.

Thus both of the above cases are impossible.  Hence, Claim \ref{spectral-spanning-tree-claim3} holds.
\q

\par\medskip
By Claim \ref{spectral-spanning-tree-claim3}, we have $|V_1|=|V_2|=\cdots=|V_{p-1}|=1$.
Hereafter, let
$$
\text{ $U_1=\{u:\,\,d_G(u)=n-1$ and $u\in V_p\}$, $U_2=V_p\setminus U_1$ and $V(G)\setminus V_p=\cup_{i=1}^{p-1}V_i=U_3$.}
$$
\begin{Claim}\label{spectral-spanning-tree-claim4}
 $p=|\pi|= 2s-2k+3$.
\end{Claim}
\p
By Claim \ref{spectral-spanning-tree-claim3} and $G$ is $s$-edge-connected, we have $k(p-1)-1=e_G(\pi)\ge (p-1)s-e(G[U_3])$. Thus $e(G[U_3])\ge  (p-1)(s-k)+1$. Combining this with  $e(G[U_3])\le \binom{p-1}{2}$,
\begin{align}\label{spectral-spanning-tree-eq9}
 (p-1)(s-k)+1\le e(G[U_3])\le {p-1\choose 2}.
\end{align}
Hence,  $p>2s-2k+2$. This, together with Claim \ref{spectral-spanning-tree-claim2}, yields that $p=|\pi|=2s-2k+3$. This proves Claim \ref{spectral-spanning-tree-claim4}.
\q

\begin{Claim}\label{spectral-spanning-tree-claim5}
Each of the following holds:
\\
(i)  $G[V_p]$ is a complete graph with  $|V_p|=n-p+1=n-2s+2k-2$;
\\
(ii)  $|U_1|=2k-s-1$;
\\
(iii)  $K_{2k-s-1}\vee (K_{n-s-1}+(2s-2k+2)K_1)\subseteq G$.
\end{Claim}
\p
As $G$ is a graph with the maximum spectral radius  among all $s$-edge-connected graphs satisfying $\tau(G)\le k-1$, by the choice of $\pi$, we have $G[V_i]$ is a complete graph for any $i\in \{1,2,\ldots,p\}$. Combining this with Claims \ref{spectral-spanning-tree-claim3} and \ref{spectral-spanning-tree-claim4}, we have $G[V_p]$ is a complete graph with  $|V_p|=n-p+1=n-2s+2k-2$. This proves Claim \ref{spectral-spanning-tree-claim5}(i).

Since $p=2s-2k+3$, if  Claim \ref{spectral-spanning-tree-claim5}(i)--(ii) hold, then   by Claims \ref{spectral-spanning-tree-claim3} and  \ref{spectral-spanning-tree-claim4}, we must have
\[
K_{2k-s-1}\vee \bigl(K_{n-s-1}+(2s-2k+2)K_1\bigr) \subseteq G,
\]
which confirms that  Claim \ref{spectral-spanning-tree-claim5}(iii) follows. Thus, to complete the proof, it suffices to show that Claim \ref{spectral-spanning-tree-claim5}(ii) holds.

By contradiction and assume that  $|U_1|\neq 2k-s-1$. We discuss the following two cases.
\\ \noindent
\textbf{Case 1.} $|U_1|\le 2k-s-2$. Since $G$ is $s$-edge-connected, for every vertex $w\in U_3$,
we have
\begin{align}\label{SPEC-EDST-eq1}
d_{V_p}(w)=|N_G(w)\cap V_p|\ge s-(|U_3|-1)=s-p+2 = 2k-s-1.
\end{align}
Combining this with $|U_1|\le 2k-s-2$, we get that there exists a vertex $u\in U_2 $ such that $U_3 \cap N_G(u)\ne \emptyset$.
Among all such vertices, choose $u\in  U_2 $ such that
$d_{U_3}(u)$ is as large as possible.
Set
\[
W_1=V(G)\setminus (V_p\cup N_G(u)).
\]
Then $\emptyset\ne W_1\subsetneqq U_3$. Thus by (\ref{SPEC-EDST-eq1}),  we have for every $w_1\in W_1$,  $d_{V_p}(w_1)\ge 2k-s-1$. Together with $|U_1|\le 2k-s-2$, this yields
\[
\big(\cup_{w_1\in W_1}N_G(w_1)\big)\cap
\bigl(U_2\setminus  \{u\}\bigr)\neq \emptyset.
\]
Assume that $u_2\in \big(\cup_{w_1\in W_1}N_G(w_1)\big)\cap
\bigl(U_2\setminus  \{u\}\bigr)$. By Lemma \ref{EDST-connect-le1}, we get $G_1=G(u,u_2)$  is $s$-edge-connected. Furthermore, by Lemma \ref{le2.1} and $G_1\not\cong G$ by the choose of $u$, we have
$\lambda(G_1)>\lambda(G)$,
contradicting the choice of $G$.
\\ \noindent
\textbf{Case 2.} $|U_1|\ge 2k-s$.  Suppose that the $i$-th vertex of $U_3$ is adjacent to exactly  $r_i$ vertices of $ U_2$ and $t_i$ vertices of $U_3$, where $i\in \{1,2,\ldots,p-1\}$.
 As $G$ is $s$-edge-connected,   we have
\begin{align*}
e_G(\pi) =& e(U_1, U_3) + e(U_2, U_3)+e(G[U_3])
\\
=& |U_1|(p-1)+\sum_{1\le i\le p-1} r_i+\frac{1}{2}\sum_{1\le i\le p-1} t_i
\\
\ge& |U_1|(p-1)+\sum_{1\le i\le p-1} r_i+\frac{1}{2}\sum_{1\le i\le p-1} (s-|U_1|-r_i)
\\
\ge& |U_1|(p-1)+\frac{1}{2}\sum_{1\le i\le p-1} (s-|U_1|)
\\
= &  \frac{1}{2}(p-1)(|U_1|+s)\ge k(p-1),
\end{align*}
contrary to (\ref{spectral-spanning-tree-eq1}).

Thus both of the above cases are impossible and so  $|U_1| = 2k-s-1$.
This proves Claim \ref{spectral-spanning-tree-claim5}.\q

\par\medskip
 From Claim \ref{spectral-spanning-tree-claim5}, we have $|U_1\cup U_3|=s+1$. Now, given  $w\in U_3$. Since $G$ is $s$-edge-connected, we have $d_{U_2}(w)+d_{U_3}(w)+d_{U_1}(w)=d_G(w)\ge s$, which implies that $d_{U_2}(w)\ge d_{\overline{G}[U_3]}(w)$.
Moreover, when $d_{U_2}(w)=0$, we have $N_G[w]=U_1\cup U_3$. Let $W_2=\{u_1,u_2,\ldots,u_{t_1}\}\subseteq U_2$ satisfies   $d_{U_3}(u_1)\ge d_{U_3}(u_2)\ge \cdots \ge d_{U_3}(u_{t_1})\ge 1$, and let $W_3=\{w_1,w_2,\ldots,w_{t_2}\}\subseteq U_3$  satisfies that   $d_{U_2}(w_1)\ge d_{U_2}(w_2)\ge \cdots \ge d_{U_2}(w_{t_2})\ge 1$.
Particularly, if there exists $i\in \{1,2,\ldots,t_2\}$ such that $d_{U_2}(w_i)= d_{U_2}(w_{i+1})$, then $W_3$  satisfies that $x_{w_{i}}\ge x_{w_{i+1}}$.

\begin{Claim}\label{spectral-spanning-tree-claim8}
Each of the following holds:
\\
(i) $N_{U_3}(u_1)\supseteq N_{U_3}(u_2)\supseteq \cdots \supseteq N_{U_3}(u_{t_1})$;
\\
(ii) $N_{U_2}(w_1)\supseteq N_{U_2}(w_2)\supseteq \cdots \supseteq N_{U_2}(w_{t_2})$, where  $t_1=d_{U_2}(w_1)$  and  $t_2=d_{U_3}(u_1)$.
\end{Claim}
\p
On the contrary, suppose that there exist $i,j\in \{1, 2, \ldots, t_1\}$ with $i<j$  such that $N_{U_3}(u_i)\nsupseteq N_{U_3}(u_j)$. As $d_G(u_i)\ge d_G(u_j)$, it follows that    $N_{U_3}(u_j)\setminus  N_{U_3}(u_i)\neq \emptyset$ and $N_{U_3}(u_i)\setminus  N_{U_3}(u_i)\neq \emptyset$.
Then $\tau(G(u_i,u_j))\le k-1$ and $G(u_i,u_j)\ncong G$. Since $G(u_i,u_j)$ is $s$-edge-connected by Lemma \ref{EDST-connect-le1}, it follows from Lemma \ref{le2.1}  that $\lambda(G(u_i,u_j))>\lambda(G)$, contrary to the choice of $G$.
So Claim \ref{spectral-spanning-tree-claim8}(i) holds.

We prove $N_{U_2}(w_1)\supseteq N_{U_2}(w_2)\supseteq \cdots \supseteq N_{U_2}(w_{t_2})$ by contradiction. Suppose that there exist $i, j \in \{1,2,\dots,t_2\}$ with $i<j$, such that $N_{U_2}(w_j) \nsubseteq N_{U_2}(w_{i})$. Then there exists $u_{j'}\in N_{U_2}(w_j) \setminus  N_{U_2}(w_{i})$. Since $d_{U_2}(w_i)\ge  d_{U_2}(w_{j})$,   there exists $u_{i'}\in N_{U_2}(w_i) \setminus  N_{U_2}(w_{j})$.
If $j'<i'$, then by Claim \ref{spectral-spanning-tree-claim8}(i), we have
$w_i\in N_G[u_{i'}]\subseteq N_G[u_{j'}]$, contrary to  $u_{j'}\in N_{U_2}(w_j)  \setminus  N_{U_2}(w_{i})$. Thus we must have $j'>i'$, then by Claim \ref{spectral-spanning-tree-claim8}(i), we have
$w_j\in N_G[u_{j'}]\subseteq N_G[u_{i'}]$, contrary to  $u_{i'}\in N_{U_2}(w_i)  \setminus  N_{U_2}(w_{j})$.  Therefore, $N_{U_2}(w_1)\supseteq N_{U_2}(w_2)\supseteq \cdots \supseteq N_{U_2}(w_{t_2})$.
 Combining this with Claim \ref{spectral-spanning-tree-claim8}(i), we have $t_1=d_{U_2}(w_1)$ and $t_2=d_{U_3}(u_1)$. So Claim \ref{spectral-spanning-tree-claim8}(ii) holds.
\q


\begin{Claim}\label{spectral-spanning-tree-claim6}
Each of the following holds:
\\
(i)  $e(U_2, U_3)+e(G[U_3])=\binom{p-1}{2}+s-k$;
\\
(ii)  $\binom{p-1}{2}-(s-k)\le e(G[U_3])\le \binom{p-1}{2}$ and $s-k\le e(U_2, U_3)\le 2s-2k$.
\end{Claim}
\p
By Claims \ref{spectral-spanning-tree-claim3} and  \ref{spectral-spanning-tree-claim4}, we have
$
e(U_2, U_3)+e(G[U_3])=e_G(\pi)-e(U_1, U_3).
$
Combining this with (\ref{spectral-spanning-tree-eq1}),
Claim \ref{spectral-spanning-tree-claim4}
and Claim \ref{spectral-spanning-tree-claim5}(ii), we have
\begin{align*}
 e(U_2, U_3)+e(G[U_3])=k(p-1)-1-(2k-s-1)(p-1)
 =\binom{p-1}{2}+s-k.
\end{align*}
So Claim \ref{spectral-spanning-tree-claim6}(i) holds.

By (\ref{spectral-spanning-tree-eq9}) and Claim \ref{spectral-spanning-tree-claim4}, we get $\binom{p-1}{2}-(s-k)\le e(G[U_3])\le \binom{p-1}{2}$. This, together with
Claim \ref{spectral-spanning-tree-claim6}(i), yields that
$s-k\le e(U_2, U_3)\le 2s-2k$. This proves Claim \ref{spectral-spanning-tree-claim6}(ii). Hence Claim \ref{spectral-spanning-tree-claim6} holds.
\q

\par\medskip
Let $W_4=\{w\,| \,w\in U_3 \text{ and } d_{U_2}(w)=\emptyset\}$.  By  Claim \ref{spectral-spanning-tree-claim6}(ii), Claims \ref{spectral-spanning-tree-claim3} and \ref{spectral-spanning-tree-claim4},
we have
\begin{equation}\label{spectral-spanning-tree-eq11-0}
|W_4|\ge |U_3|-(2s-2k)=p-1-2s+2k=2.
\end{equation}
From Claim \ref{spectral-spanning-tree-claim5}, we have $|U_1\cup U_3|=s+1$. Combining this with $G$ is $s$-edge-connected,   we get
\begin{equation}\label{spectral-spanning-tree-eq11}
\text{for any $w\in W_4$, $d_G(w)=s$ and $N_G[w]=U_1\cup U_3$.}
\end{equation}

As $A(G)\mathbf{x}=\lambda(G)\mathbf{x}$,  by Claim \ref{spectral-spanning-tree-claim5}(i) and (ii),  for any $u\in U_1$, $z\in U_2$ and $w\in U_3$, we have
\begin{align}
\lambda(G)x_{u}&=\sum_{v\in U_1\setminus \{u\}}x_v+\sum_{v\in U_2}x_v+\sum_{v\in U_3}x_v,
\label{spectral-spanning-tree-eq13}
\\
\lambda(G)x_{z}&=\sum_{v\in U_1}x_v+\sum_{v\in U_2\setminus  \{z\}}x_v+\sum_{v\in N_{U_3}(z)}x_v,
\label{spectral-spanning-tree-eq14}
\\
\lambda(G)x_{w}&=\sum_{v\in U_1}x_v+\sum_{v\in N_{U_3}(w)}x_v+\sum_{v\in N_{U_2}(w)}x_v. \label{spectral-spanning-tree-eq15}
\end{align}

For short, we write $\lambda$ for $\lambda(G)$. Let $x_{z_0}=\min\{x_{v}: v\in U_2\}$,  $x_{w_0}=\max\{x_{v}: v\in U_3\}$ and $x_{w_0'}=\min\{x_{v}: v\in U_3\}$. By Claim  \ref{spectral-spanning-tree-claim3}, Claim \ref{spectral-spanning-tree-claim8}(ii) and Claim \ref{spectral-spanning-tree-claim6}(ii),    we have for any $w\in U_3$,
\begin{align}\label{spec-EDST-eq2}
p-2-(s-k)\le d_{U_3}(w)\le p-2  \text{~and~}  0\le d_{U_2}(w)\le 2(s-k).
\end{align}
Combining (\ref{spec-EDST-eq2}) with Claim \ref{spectral-spanning-tree-claim4}, Claim \ref{spectral-spanning-tree-claim6}(ii), (\ref{spectral-spanning-tree-eq14}) and (\ref{spectral-spanning-tree-eq15}), we have
\begin{align*}
\lambda x_{z_0}-\lambda x_{w_0}
= &\sum_{v\in U_2\setminus  \{z_0\}}x_v+\sum_{v\in N_{U_3}(z_0)}x_v -\sum_{v\in N_{U_2}(w_0)}x_v-\sum_{v\in N_{U_3}(w_0)}x_v
\\
\ge &  \sum_{v\in U_2\setminus  \{z_0\}}x_v-\sum_{v\in N_{U_2}(w_0)}x_v -\sum_{v\in N_{U_3}(w_0)}x_v
\\\nonumber
=&\sum_{v \in U_2 \setminus \big(\{z_0\} \cup N_G(w_0)\big)} x_v -\sum_{v\in N_{U_3}(w_0)}x_v
\\\ge &\big(n - s - 1 - 1 - 2(s - k)\big)x_{z_0} - (p - 2)x_{w_0}
\nonumber \\
= & (n-3s+2k-2)x_{z_0} - (2s-2k+1)x_{w_0},
\end{align*}
which implies that
\begin{align}\label{spectral-spanning-tree-eq16}
(\lambda-n+3s-2k+2)x_{z_0}\ge (\lambda-2s+2k-1)x_{w_0}.
\end{align}
As $n-1>\lambda> n-2s+2k-3$, $2\le k\le s\le 2k-2$ and $n\ge 3s^2$, we have $\lambda-2s+2k-1>\lambda-n+3s-2k+1>0$.
Combining this with (\ref{spectral-spanning-tree-eq16}), we have
\begin{align}\label{spectral-spanning-tree-eq17}
x_{z_0}\ge \frac{\lambda-2s+2k-1}{\lambda-n+3s-2k+2}x_{w_0}> \frac{n-2s+2k-2}{3s-2k+1}x_{w_0}> x_{w_0},
\end{align}
where the second inequality holds as $(\lambda-2s+2k-1)/(\lambda-n+3s-2k+2)$ is decreasing on $\lambda$  for $n-2s+2k-3<\lambda<n-1.$

Let $W_5=\{u\,|\,u\in U_2 \text{ and } N_{U_3}(u)=\emptyset\}$, that is,   $W_5=U_2\setminus W_2$.
By (\ref{spectral-spanning-tree-eq14}), we have for any $u\in W_5$ and $z\in W_2$,
\begin{align} \label{spectral-spanning-tree-eq18-1}
\lambda x_{u} - \lambda x_{z}= x_{z} - x_{u} - \sum_{v\in N_{U_3}(z) }x_v.
\end{align}
As $d_{U_3}(z)\le t_2$ for any  $z\in W_2$, by  (\ref{spectral-spanning-tree-eq18-1}), we have
\begin{align}\label{spectral-spanning-tree-eq18-2}
0> x_{u} - x_{z} \ge -\frac{d_{U_3}(z) x_{w_0} }{\lambda + 1} \ge -\frac{t_2 x_{w_0} }{\lambda + 1}.
\end{align}

\begin{Claim}\label{spectral-spanning-tree-claim7}
 Each of the following holds:
\\
(i) $e(U_2, U_3)= 2s-2k$ and $e(G[U_3])=\binom{p-1}{2}-s+k$;
\\
(ii) for any $v\in U_3$,  $d_G(v)=s$ and $d_{\overline{G}[U_3]}(v)=d_{U_2}(v)$;
\\
(iii)  $e(\overline{G}[U_3])=s-k$ and $e(G)=\binom{n-p+1}{2}+2k(s-k+1)-1$;
\\
(iv) $n-2s+2k-3 < \lambda < n-2s+2k-2$;
\\
 (v) $d_{\overline{G}[U_3]}(w_1)\ge d_{\overline{G}[U_3]}(w_2)\ge \cdots \ge d_{\overline{G}[U_3]}(w_{t_2})$.
\end{Claim}
\p
We prove $e(U_2, U_3)= 2s-2k$ by contradiction.
Assume that $e(U_2, U_3)\neq 2s-2k$.  By Claim \ref{spectral-spanning-tree-claim6}(ii), we have $s-k\le e(U_2, U_3)\le 2s-2k-1$.  Let $e(U_2, U_3)=s-k+r$, where $0\le r\le s-k-1$.
By Claim \ref{spectral-spanning-tree-claim6}(i), we get $e(G[U_3])=\binom{p-1}{2}-r\ge \binom{p-1}{2}-s+k+1$.

We next prove there exists $w\in U_3$ such that  $d_G(w)\ge s+1$.
Otherwise, as $G$ is s-edge-connected, we have $d_G(w)= s$ for any $w\in U_3$. This, together with   $e(G[U_3])=\binom{p-1}{2}-r$, Claims \ref{spectral-spanning-tree-claim3}, \ref{spectral-spanning-tree-claim4} and \ref{spectral-spanning-tree-claim5}(ii), yields that
\begin{align*}
e(U_2, U_3)=&  s|U_3|-2e(G[U_3])-e(U_1, U_3)
\\
=& s(p-1)-2\binom{p-1}{2}+ 2r -(2k-s-1)(p-1)=
2r.
\end{align*}
Since $e(U_2, U_3)=s-k+r$, it follows that $r= s-k$, contrary to $0\le r\le s-k-1$. Hence, there exists $w\in U_3$ such that  $d_G(w)\ge s+1$.
Since $|U_1\cup U_3|=s+1$, we get $d_{U_1}(w)+d_{U_3}(w)\le s$. Then there exists $z\in U_2$ such that $z\in N_G(w)$.

By (\ref{spectral-spanning-tree-eq11-0}) and (\ref{spectral-spanning-tree-eq11}),  we get that there exist $w'\in U_3$ such that $w'w\in E(G[U_3])$ and  $N_G(w')\cap U_2=\emptyset$. Let $G_1=G-w'w+zw'$, where $z\in U_2\cap N_G(w)$. Then $G_1$ is $s$-edge-connected by Lemma \ref{EDST-connect-le1} and $\tau(G_1)\le k-1$ by \eqref{spectral-spanning-tree-eq1}. Moreover,
by (\ref{spectral-spanning-tree-eq17}), we get $x_z\ge x_{z_0}> x_{w_0}\ge x_{w}$. Combining this with Lemma \ref{le2.2}, we have $\lambda(G_1)>\lambda(G)$, contrary to the choice of $G$.
Thus $e(U_2, U_3)= 2s-2k$, and by Claim \ref{spectral-spanning-tree-claim6}(i), we have $e(G[U_3])=\binom{p-1}{2}-s+k$.
So Claim \ref{spectral-spanning-tree-claim7}(i) holds.

 We next show that  for any $v\in U_3$,  $d_G(v)=s$.
By contradiction, assume that there exists $w\in U_3$ such  that $d_G(w)\ge s+1$. As $G$ is $s$-edge-connected, by Claims \ref{spectral-spanning-tree-claim5}(ii) and    \ref{spectral-spanning-tree-claim7}(i),
\begin{align*}
 s+1+(p-2)s\le & \sum_{v\in U_3}d_G(v)
 \\
 = & e(U_1, U_3) + e(U_2, U_3)+2e(G[U_3])
 \\
 = &  (2k-s-1)(p-1) +  2s-2k + 2\binom{p-1}{2}-2s+2k=s(p-1),
\end{align*}
 a contradiction. Therefore, $d_G(v)=s$ holds  for any $v\in U_3$. This,   together with $|U_3 \cup U_1 |=s+1$, yields that $d_{\overline{G}[U_3]}(v)=d_{U_2}(v)$.
This proves Claim \ref{spectral-spanning-tree-claim7}(ii).

Claim \ref{spectral-spanning-tree-claim7}(iii)  follows immediately from Claims \ref{spectral-spanning-tree-claim3},  \ref{spectral-spanning-tree-claim5}  and \ref{spectral-spanning-tree-claim7}(i).

As $G\in\mathcal{H}_1(s,k)$, by Lemma \ref{EDST-connect-le2}, Claims \ref{spectral-spanning-tree-claim4} and \ref{spectral-spanning-tree-claim7}(iii), we have Claim \ref{spectral-spanning-tree-claim7}(iv)   holds.

Claim \ref{spectral-spanning-tree-claim7}(v)  follows immediately from
 Claims \ref{spectral-spanning-tree-claim8}(ii) and \ref{spectral-spanning-tree-claim7}(ii).
This proves Claim \ref{spectral-spanning-tree-claim7}.
\q

By (\ref{spectral-spanning-tree-eq17}) and $\lambda<n-2s+2k-2$ in Claim \ref{spectral-spanning-tree-claim7}(iv), we have
\begin{align}\label{spectral-spanning-tree-eq17-1}
x_{z_0}\ge \frac{\lambda-2s+2k-1}{\lambda-n+3s-2k+2}x_{w_0}> \frac{n-4s+4k-3}{s}x_{w_0}> x_{w_0},
\end{align}
where the second inequality holds as $(\lambda-2s+2k-1)/(\lambda-n+3s-2k+2)$ is decreasing on $\lambda$.

Recall that for any $w\in W_4$, $N_{U_3}[w]=U_3$  by (\ref{spectral-spanning-tree-eq11}).
For any $w'\in  U_3 \setminus W_4$, by the definition of $W_4$, we have $N_G(w')\cap U_2 \neq \emptyset$.  This, together with (\ref{spectral-spanning-tree-eq15}), Claim \ref{spectral-spanning-tree-claim7}(ii) and $G$ is $s$-edge-connected, yields that
for any $w'\in U_3\setminus W_4$ and $w\in W_4$,
\begin{align*}
\lambda x_{w'} - \lambda x_{w} &=\sum_{v\in N_{U_2}(w')}x_v+
\sum_{v\in N_{U_3}(w')}x_v-\sum_{v\in N_{U_3}(w)}x_v
\nonumber \\
&=\sum_{v\in N_{U_2}(w')}x_v +
\sum_{v\in N_{U_3}(w')}x_v
-\sum_{v\in U_3}x_v +x_w
\nonumber \\
& \ge d_{\overline{G} [U_3]}(w') x_{z_0} + x_w - x_{w'} - d_{\overline{G}\left[U_3\right]}(w') x_{w_0},
\end{align*}
Combining this with $x_{z_0}> x_{w_0}$ in \eqref{spectral-spanning-tree-eq17-1}, we have
\begin{align}\label{spectral-spanning-tree-eq18}
x_{w'} - x_{w} \ge \frac{d_{\overline{G}[U_3]}(w')(x_{z_0}-x_{w_0})}{\lambda + 1}>0.
\end{align}
Then by (\ref{spectral-spanning-tree-eq18}), Claim \ref{spectral-spanning-tree-claim7}(v)  and $N_G[w]=U_1\cup U_3$ for any $w\in W_4$  in (\ref{spectral-spanning-tree-eq11}), we have
\begin{equation}\label{spectral-spanning-tree-eq18-0}
\text{$x_{w_0}=x_{w_1}$  and $x_{w_0'}=x_{w}$ for any $w\in W_4$}.
\end{equation}

\begin{Claim}\label{spectral-spanning-tree-claim8-1}
If there exists $i\in \{1,2,\ldots,t_2\}$  such that $N_{U_2}(w_i)\supsetneqq N_{U_2}(w_{i+1})$, then $x_{w_{i}}>x_{w_{i+1}}$.
\end{Claim}
\p
Assume that there exists $i\in \{1,2,\ldots,t_2\}$  such that $N_{U_2}(w_i)\supsetneqq N_{U_2}(w_{i+1})$. By Claim \ref{spectral-spanning-tree-claim7}(i), we have $\triangle({\overline{G}[U_3]})\le s-k$. As $n\ge 3s^2$,
by (\ref{spectral-spanning-tree-eq15}), (\ref{spectral-spanning-tree-eq17-1}), Claims \ref{spectral-spanning-tree-claim5}(iii) and \ref{spectral-spanning-tree-claim7}(i), we have
\begin{align*}
 \lambda x_{w_i}- \lambda x_{w_{i+1}} &= \sum_{v\in N_{U_2}(w_i) \setminus  N_{U_2}(w_{i+1}) } x_v
 + \sum_{v\in N_{U_3}(w_i)}x_v - \sum_{v\in N_{U_3}(w_{i+1})}x_v
\\&=\sum_{v\in N_{U_2}(w_i) \setminus  N_{U_2}(w_{i+1})} x_v
+ \sum_{v\in U_3} x_v-\sum_{v\in N_{\overline{G}[U_3]}[w_i]} x_v
- \Big(\sum_{v\in U_3} x_v-\sum_{v\in N_{\overline{G}[U_3]}[w_{i+1}]}x_v\Big)  \\&\ge \big|N_{U_2}(w_i)\setminus N_{U_2}(w_{i+1})\big|x_{z_0}
-\sum_{v\in N_{\overline{G}[U_3]}(w_i)} x_v
-x_{w_i}+x_{w_{i+1}} \\
&\ge x_{z_0} - \sum_{v\in N_{\overline{G}[U_3]}(w_i)} x_v -x_{w_i}+x_{w_{i+1}}\\
&\ge x_{z_0} - (s - k)x_{w_0}-x_{w_i}+x_{w_{i+1}} \\
&> \frac{n - 4s + 4k - 3}{s} x_{w_0} - (s - k)x_{w_0} -x_{w_i}+x_{w_{i+1}}\\
&= \frac{n-s^2+sk-4s+4k-3}
     {s}x_{w_0}-x_{w_i}+x_{w_{i+1}} \\&>-x_{w_i}+x_{w_{i+1}}.
\end{align*}
Thus we have  $x_{w_i}>x_{w_{i+1}}$.  So Claim \ref{spectral-spanning-tree-claim8-1}   holds.
\q

By Claim \ref{spectral-spanning-tree-claim7}(i) and (v), we have $d_{\overline{G}[U_3]}(w_1)=
\triangle({\overline{G}[U_3]})\le s-k$.
This, together with Claims \ref{spectral-spanning-tree-claim8}(ii) and \ref{spectral-spanning-tree-claim7}(ii), yields that
$t_1=d_{U_2}(w_1)=d_{\overline{G}[U_3]}(w_1)
=\triangle({\overline{G}[U_3]})\le s-k$ and
 $ t_1+1=\triangle({\overline{G}[U_3]})+1 \le |W_3|=t_2=d_{U_3}(u_1)\le e(U_2, U_3)= 2s-2k$. Thus
\begin{align}\label{spectral-spanning-tree-eq19}
\begin{cases}
 t_1= s-k \text{~and~} G=H_1(s,k), & \text{if } 0\le s-k\le 1;  \\
1 \le t_1\le s-k \text{~and~} t_1+1\le t_2\le 2s-2k, & \text{if }  s-k\ge 2.
\end{cases}
\end{align}

\begin{Claim}\label{spectral-spanning-tree-claim9}
Assume that $s-k\ge 2$. Each of the following holds:
\\
(i) for every $r\in \{1,2,\ldots,t_2\}$, if there exist two distinct integers
$i, j\in \{1,2,\ldots,t_2\}$ with $r<\min\{i,j\}$ such that $w_iw_j\not\in E(G[U_3])$, then $ w_i,w_j\not\in  N_{U_3}(w_r)$;
\\
(ii)  $W_3=N_{\overline{G}[U_3]}[w_1]  \supseteq   N_{\overline{G}[U_3]}[w_2]$ and   $N_{\overline{G}[U_3]}[w_i]\supseteq N_{\overline{G}[U_3]}(w_{i+1})$ for any $i\in \{2,3,\ldots,t_2-1\}$.
\end{Claim}

\p
By contradiction, assume that $ |\{w_i,w_j\}\cap  N_{U_3}(w_r)|\ge 1$.
Without loss of generality, assume that  $w_iw_r\in E(G)$.
By  Claim \ref{spectral-spanning-tree-claim8}(i) and \ref{spectral-spanning-tree-claim7}(ii), we have $u_1w_j\in E(G)$.
As  $n\ge 3s^2$ and $2\le k\le s\le 2k-2$, we have $|U_2|=n-s-1\ge 3s^2-s+1\ge 2s-2k+1$.  Combining this with  $e(U_2, U_3)= 2s-2k$ in  Claim \ref{spectral-spanning-tree-claim7}(i),   there exists $z'\in U_2\setminus   W_2$ such that $z'w_r\notin E(G)$. Let $G_1=G+z'w_r-w_rw_i+w_iw_j-w_ju_1$.
Then $G_1$ is $s$-edge-connected  by Lemma \ref{EDST-connect-le1} and $\tau(G_1)\le k-1$.

Let $\mathbf y$ be the Perron vector of $A(G_1)$.
As $A(G_1)\mathbf{y}=\lambda(G_1)\mathbf{y}$,  for any $u\in U_1$, $z\in U_2$ and $w\in U_3$, we have
\begin{align}
\lambda(G_1)y_{u}&=\sum_{v\in U_1\setminus \{u\}}y_v+\sum_{v\in U_2}y_v+\sum_{v\in U_3}y_v,
\label{spectral-spanning-tree-eq20}
\\
\lambda(G_1)y_{z}&=\sum_{v\in U_1}y_v+\sum_{v\in U_2\setminus  \{z\}}y_v+\sum_{v\in N_{G_1}(z)\cap U_3}y_v,
\label{spectral-spanning-tree-eq21}
\\
\lambda(G_1)y_{w}&=\sum_{v\in U_1}y_v+\sum_{v\in N_{G_1}(w)\cap U_3}y_v+\sum_{v\in N_{G_1}(w)\cap U_2}y_v. \label{spectral-spanning-tree-eq22}
\end{align}

For short, we  write $\lambda_1$ for $\lambda(G_1)$. Let $y_{z_0^1}=\min\{y_{v}: v\in U_2\}$ and $y_{w_0^1}=\max\{y_{v}: v\in U_3\}$.
By Claim \ref{spectral-spanning-tree-claim7}(v) and $e(\overline{G}[U_3])=s-k$ in Claim \ref{spectral-spanning-tree-claim7}(iii), we
 have $\triangle({\overline{G}[U_3]})\le s-k-1$ by our hypothesis on the existence of $r$. By the construction of $G_1$, we have $\triangle({\overline{G_1}[U_3]})\le s-k.$
Combining this with (\ref{spec-EDST-eq2}),  we have for any  $w\in U_3$,
\begin{align}\label{spec-EDST-eq3}
p-2-(s-k)\le  d_{G_1[U_3]}(w)\le p-2  \text{~and~} 0\le  d_{G_1[U_2]}(w)\le s-k.
\end{align}
This, together with (\ref{spectral-spanning-tree-eq21}), (\ref{spectral-spanning-tree-eq22}), Claims  \ref{spectral-spanning-tree-claim3} and \ref{spectral-spanning-tree-claim4},   we have
\begin{align*}
\lambda_1y_{z_0^1}-\lambda_1y_{w_0^1} &= \sum_{v\in U_2\setminus \{z_0^1\}}y_v+\sum_{v\in N_{G_1}(z_0^1)\cap U_3}y_v
 -\sum_{v\in N_{G_1}(w_0^1)\cap U_3}y_v-\sum_{v\in N_{G_1}(w_0^1)\cap U_2}y_v
\\
&\ge \sum_{v\in U_2\setminus  \{z_0^1\}}y_v - \sum_{v\in N_{G_1}(w_0^1)\cap U_2}y_v
- \sum_{v\in N_{G_1}(w_0^1)\cap U_3}y_v
\nonumber \\
&\ge \sum_{v\in U_2\setminus  N_{G_1}(w_0^1)}y_{z_0^1} - y_{z_0^1} -(p-2)y_{w_0^1}
\nonumber \\
 &\ge (|U_2|-(s-k))y_{z_0^1}-y_{z_0^1}-(2s-2k+1)y_{w_0^1}
\nonumber \\
 &= (n-2s+k-2)y_{z_0^1}-(2s-2k+1)y_{w_0^1} ,
\end{align*}
which implies that
\begin{align}\label{spectral-spanning-tree-eq23}
(\lambda_1-n+2s-k+2)y_{z_0^1}\ge (\lambda_1-2s+2k-1)y_{w_0^1}.
\end{align}
By Claim \ref{spectral-spanning-tree-claim5}(iii) and $G_1=G+z'w_r-w_rw_i+w_iw_j-w_ju_1$, we have  $K_{n-2s+2k-2}\subsetneqq G_1$.
Thus $n-1>\lambda(G_1)=\lambda_1 > n-2s+2k-3$. This, together with $4\le k+2\le s\le 2k-2$, $n\ge 3s^2$ and  (\ref{spectral-spanning-tree-eq23}), yields that
\begin{align}\label{spectral-spanning-tree-eq24}
y_{z_0^1}\ge \frac{\lambda_1-2s+2k-1}{\lambda_1-n+2s-k+2}y_{w_0^1}\ge \frac{n-2s+2k-2}{2s-k+1}y_{w_0^1}> y_{w_0^1}.
\end{align}

As $G_1=G+z'w_r-w_rw_i+w_iw_j-w_ju_1$,  we have $d_{\overline{G_1}[U_3]}(w_r)=
d_{\overline{G}[U_3]}(w_r)+1$ and $d_{G_1[U_2]}(w_r)=d_{U_2}(w_r)+1$. As $d_{\overline{G}[U_3]}(w_r)=d_{U_2}(w_r)$
in Claims \ref{spectral-spanning-tree-claim7}(ii), we have  $d_{\overline{G_1}[U_3]}(w_r)=d_{G_1[U_2]}(w_r)$.
Combining this with  (\ref{spectral-spanning-tree-eq22}), (\ref{spec-EDST-eq3}) and  Claim \ref{spectral-spanning-tree-claim8}(ii),
 we have
\begin{align*}
&\lambda_1y_{w_r} - \lambda_1y_{w_j}
\nonumber\\
=&
\sum_{v\in N_{G_1}(w_r)\cap   U_2}y_v
  - \sum_{v\in N_{G_1}(w_j)\cap U_2}y_v
  +\sum_{v\in N_{G_1}(w_r)\cap U_3}y_v
  - \sum_{v\in N_{G_1}(w_j)\cap U_3}y_v
  \nonumber\\
=&
\sum_{v\in N_{G_1}(w_r)\cap U_2}y_v
  - \sum_{v\in N_{G_1}(w_j)\cap  U_2}y_v
  +\big(\sum_{v\in U_3}y_v
  -\sum_{v\in N_{{\overline{G_1}}[U_3]}[w_r]}y_v\big)
   - \big(\sum_{v\in U_3}y_v
  -\sum_{v\in N_{{\overline{G_1}}[U_3]}[w_j]}y_v\big)
   \nonumber\\
=&
\sum_{v\in N_{G_1}(w_r)\cap U_2}y_v
  - \sum_{v\in N_{G_1}(w_j)\cap U_2}y_v
    +\sum_{v\in N_{{\overline{G_1}}[U_3]}[w_j]}y_v
  -\sum_{v\in N_{{\overline{G_1}}[U_3]}[w_r]}y_v
\nonumber \\
\ge &
y_{z'}+y_{u_1}+ y_{w_j} -y_{w_r} -\sum_{v\in N_{{\overline{G_1}}[U_3]}(w_r)}y_v
\\
\ge &
y_{z'}+y_{u_1}+ y_{w_j} -y_{w_r} -  d_{G_1[U_2]}(w_r) y_{w_0^1}
\\
\ge & 2y_{z_0^1} -  (s-k)y_{w_0^1} + y_{w_j} -y_{w_r}.
\end{align*}
Combining this with (\ref{spectral-spanning-tree-eq24}),  $k+2\le s\le 2k-2$ and $n\ge 3s^2$, we have
\begin{equation}\label{spectral-spanning-tree-eq25}
y_{w_r} - y_{w_j}\ge \frac{2y_{z_0^1} - (s-k)y_{w_0^1}}{\lambda_1+1}\ge
\frac{2n-2s^2+3sk-k^2-5s+5k-4}{(2s-k+1)(\lambda_1+1)}y_{w_0^1} >0.
\end{equation}

Since $z'\in U_2\setminus W_2$, we have $N_G(z')\cap U_3=\emptyset$. Combining this with
 $G_1=G+z'w_r-w_rw_i+w_iw_j-w_ju_1$,  (\ref{spectral-spanning-tree-eq21}) and  Claim   \ref{spectral-spanning-tree-claim8}(i), we have
\begin{align*}
\lambda_1y_{z'} - \lambda_1y_{u_1}&= y_{u_1} - y_{z'} + \sum_{v\in N_{G_1}(z')\cap U_3}y_v  - \sum_{v\in N_{G_1}(u_1)\cap U_3}y_v
\nonumber \\
&= y_{u_1} - y_{z'} - \sum_{v\in N_{G_1}(u_1)\cap (U_3\setminus\{w_r\})}y_v
\nonumber \\
&\ge
y_{u_1}-y_{z'}-(t_2-1)y_{w_0^1}
\end{align*}
By $t_2\le 2s-2k$ in (\ref{spectral-spanning-tree-eq19}), we have
\begin{equation}\label{spectral-spanning-tree-eq26}
y_{z'} - y_{u_1}\ge \frac{1- t_2}{\lambda_1+1}y_{w_0^1} \ge \frac{1- 2s+2k}{\lambda_1+1}y_{w_0^1}.
\end{equation}

As $G_1=G+z'w_r-w_rw_i+w_iw_j-w_ju_1$ and $t_2\le 2s-2k$ in (\ref{spectral-spanning-tree-eq19}), by (\ref{spectral-spanning-tree-eq18-2}), (\ref{spectral-spanning-tree-eq17-1}) and  (\ref{spectral-spanning-tree-eq24})--(\ref{spectral-spanning-tree-eq26}),
we have
\begin{align}\nonumber
&{\mathbf y}^T (\lambda(G_1)-\lambda(G)) {\mathbf x}= {\mathbf y}^T (A(G_1)-A(G)) {\mathbf x}
\nonumber\\
=& y_{z'}  x_{w_r}+y_{w_r}x_{z'} -y_{w_r}x_{w_i}-y_{w_i}x_{w_r}
+y_{w_i}x_{w_j}+y_{w_j}x_{w_i}-y_{w_j}x_{u_1}-y_{u_1}x_{w_j}
\nonumber\\
=& (y_{z'} -y_{w_i} )(x_{w_r}-x_{w_j})+(y_{z'} -y_{u_1})x_{w_j}+
(x_{z'} -x_{w_i} )(y_{w_r}-y_{w_j})+(x_{z'} -x_{u_1})y_{w_j}
\nonumber\\
\ge& (y_{z'} -y_{u_1})x_{w_j}+
(x_{z'} -x_{w_i} )(y_{w_r}-y_{w_j})+(x_{z'} -x_{u_1})y_{w_j}
\nonumber\\
\ge& (y_{z'} -y_{u_1})x_{w_j}+
(x_{z_0} -x_{w_0} )(y_{w_r}-y_{w_j})+(x_{z'} -x_{u_1})y_{w_j}
\nonumber\\
>& \frac{1-2s+2k}{\lambda_1+1}y_{w_0^1}x_{w_j} + \frac{(2n-2s^2+3sk-k^2-5s+5k-4)(n-5s+4k-3)}{(\lambda_1+1)(2s-k+1)s}x_{w_0} y_{w_0^1}-\frac{t_2}{\lambda+1}x_{w_0}y_{w_j}
\nonumber\\\label{spec-EDST-eq6}
>& \left(-\frac{2s-2k}{\lambda_1+1} + \frac{(2n-2s^2+3sk-k^2-5s+5k-4)(n-5s+4k-3)}{(\lambda_1+1)(2s-k+1)s}-
\frac{2s-2k}{\lambda+1}\right)x_{w_0}y_{w_0^1}.
\end{align}
Since $n-1>\lambda_1 > n-2s+2k-3$, $\lambda > n-2s+2k-3$, $k+2\le s\le 2k-2$ and $n\ge 3s^2$,  we have
\begin{align*}
& -\frac{2s-2k}{\lambda_1+1} + \frac{(2n-2s^2+3sk-k^2-5s+5k-4)(n-5s+4k-3)}{(\lambda_1+1)(2s-k+1)s}-
\frac{2s-2k}{\lambda+1}
\\
>& \frac{(2n-2s^2+3sk-k^2-5s+5k-4)(n-5s+4k-3)}{sn(2s-k+1)}-
\frac{4s-4k}{ n-2s+2k-2}
\\
\ge& \frac{(2n-3k^2+5k-2)(n-6k+7)}{n(2s-k+1)s}-  \frac{2}{k}
\\
\ge& \frac{(2n-3k^2+5k-2)(n-6k+7)}{2kn(3k-3)}-  \frac{2}{k}
\\
= &  \frac{2n^2-(3k^2+19k-24)n
  +18k^3-51k^2+47k-14}{6kn(k-1)}
>0.
\end{align*}
This, together with (\ref{spec-EDST-eq6}), yields that
\begin{align*}
{\mathbf y}^T (\lambda(G_1)-\lambda(G)) {\mathbf x} >0.
\end{align*}
Thus we have $\lambda(G_1)>\lambda(G)$, a contradiction.
This proves Claim \ref{spectral-spanning-tree-claim9}(i).

By  Claims  \ref{spectral-spanning-tree-claim7}(ii) and   \ref{spectral-spanning-tree-claim9}(i), we have Claim \ref{spectral-spanning-tree-claim9}(ii)  holds. This proves Claim \ref{spectral-spanning-tree-claim9}.
\q

By  Claims \ref{spectral-spanning-tree-claim8}(ii), \ref{spectral-spanning-tree-claim7}(ii)  and \ref{spectral-spanning-tree-claim9}(ii), we have
\begin{align}\label{spectral-spanning-tree-eq27-0}
 d_{U_2}(w_1)=t_1=d_{\overline{G}[U_3]}(w_1)=t_2-1.
\end{align}

By Claim \ref{spectral-spanning-tree-claim7}(ii) and (v), we have  $d_{U_2}(w_2)\le d_{U_2}(w_1)=t_1$.
By (\ref{spectral-spanning-tree-eq18}) and Claim \ref{spectral-spanning-tree-claim8-1}, we have
\begin{align}\label{spec-EDST-add-eq1}
x_{w_0}=x_{w_1}\ge x_{w_2}\ge \cdots \ge x_{w_{t_2}}>x_{w_0'}.
\end{align}


\begin{Claim}\label{spec-span-tree-claim10-1}
Each of the following holds:
\\
(i) $d_{U_2}(w_2)\le \frac{s-k+1}{2}$;
\\
(ii) for any $w\in W_4$, \begin{align}\label{n1e}x_{w}-x_{w_2}\ge \frac{1}{\lambda + 1}x_{w_0}-\frac{s-k+1}{2(\lambda + 1)}x_{z_0}-\frac{t_2(s-k+1)}{2(\lambda + 1)^2}x_{w_0}.\end{align} Particularly, if $t_1<s-k=3$, then \begin{align}\label{n2e}x_{w}-x_{w_2}=x_{w}-x_{w_0}\ge \frac{2}{\lambda + 1}x_{w_0}-\frac{2}{\lambda + 1}x_{z_0}-\frac{6}{(\lambda + 1)^2}x_{w_0}.\end{align}
\end{Claim}
\p
By Claim \ref{spectral-spanning-tree-claim9}(ii), we have
  $ 2d_{\overline{G}[U_3]}(w_2)-1\le d_{\overline{G}[U_3]}(w_1)+
  d_{\overline{G}[U_3]}(w_2)-1\le e(\overline{G}[U_3])$.
Combining this with $e(\overline{G}[U_3])=s-k$ in Claim \ref{spectral-spanning-tree-claim7}(iii),   we have
$d_{\overline{G}[U_3]}(w_2)\le \frac{s-k+1}{2}$.
As $d_{U_2}(w_2)=d_{\overline{G}[U_3]}(w_2)$ in Claim  \ref{spectral-spanning-tree-claim7}(ii), we have
$d_{U_2}(w_2)\le \frac{s-k+1}{2}$. This proves Claim \ref{spec-span-tree-claim10-1}(i).

By (\ref{spectral-spanning-tree-eq18-2}), we have for any $i\in \{1,2,\dots,t_1\}$,
$x_{u_i}\le x_{z_0}+ \frac{t_2 x_{w_0} }{\lambda + 1}$. This, together with
 (\ref{spectral-spanning-tree-eq15}), Claims    \ref{spectral-spanning-tree-claim7}(ii)  and \ref{spec-span-tree-claim10-1}(i),
yields that for any $w\in W_4$,
\begin{align*}
\lambda(G)x_{w}-\lambda(G)x_{w_2}=&\sum_{v\in N_{U_3}(w)}x_v-\sum_{v\in N_{U_3}(w_2)}x_v-\sum_{v\in N_{U_2}(w_2)}x_v
\\
=&\sum_{v\in U_3}x_v-x_w-\sum_{v\in N_{U_3}(w_2)}x_v-\sum_{v\in N_{U_2}(w_2)}x_v
\\
\ge & x_{w_2}-x_{w}+ \sum_{v\in N_{\overline{G}[U_3]}(w_2)}x_v - d_{U_2}(w_2) \left(x_{z_0}+ \frac{t_2 x_{w_0} }{\lambda + 1}\right)
\\
\ge & x_{w_2}-x_{w}+ \sum_{v\in N_{\overline{G}[U_3]}(w_2)}x_v- \frac{s-k+1}{2}\left(x_{z_0}+ \frac{t_2 x_{w_0} }{\lambda + 1}\right),
\end{align*}
which implies that
\begin{align}\label{spec-tree-eq1}
x_{w}-x_{w_2}\ge & \sum_{v\in N_{\overline{G}[U_3]}(w_2)}\frac{x_v}{\lambda + 1}
-\frac{s-k+1}{2(\lambda + 1)}x_{z_0}-\frac{t_2(s-k+1)}{2(\lambda + 1)^2}x_{w_0}.
\end{align}
As $w_1\notin N_G(w_2)$ by Claim \ref{spectral-spanning-tree-claim9}(ii), by (\ref{spec-tree-eq1}) and  $x_{w_1}=x_{w_0} $ in (\ref{spec-EDST-add-eq1}), we have \eqref{n1e} holds.

Particularly, if $t_1<s-k=3$, then by  Claim \ref{spectral-spanning-tree-claim9}(ii) and $t_1=d_{\overline{G}[U_3]}(w_1)$ in (\ref{spectral-spanning-tree-eq27-0}), we have  $\overline{G}[W_3]\cong K_3$.
Thus $N_{\overline{G}[U_3]}[w_2]=\{w_1,w_2,w_3\}=W_3$ and $N_{U_2}(w_1)=N_{U_2}(w_2)=N_{U_2}(w_3)=\{u_1,u_2\}=W_2$ by Claim \ref{spectral-spanning-tree-claim8}(i).
Combining this with   (\ref{spectral-spanning-tree-eq15}), we have  $x_{w_1}=x_{w_2}=x_{w_3}$. Thus by (\ref{spec-tree-eq1}), we get
\eqref{n2e} holds.
So Claim \ref{spec-span-tree-claim10-1}(ii) holds. This proves Claim \ref{spec-span-tree-claim10-1}.
\q


\begin{Claim}\label{spectral-spanning-tree-claim10}
  $t_1=s-k$.
\end{Claim}

\p
If $0\le s-k\le 2$, then by Claims  \ref{spectral-spanning-tree-claim8}(ii) and \ref{spectral-spanning-tree-claim9}(ii),
we have $t_1=s-k$. Thus we assume that $s-k\ge 3$.
On the contrary, suppose that $t_1\neq s-k$ when $s-k\ge 3$.
As $t_1\neq s-k$, by (\ref{spectral-spanning-tree-eq19}) and  (\ref{spectral-spanning-tree-eq27-0}), we have
\begin{align}\label{spectral-spanning-tree-eq27}
 1\le t_1<s-k \text{~and~}  t_2=t_1+1\le s-k.
\end{align}
This, together with Claims  \ref{spectral-spanning-tree-claim7}(iii) and \ref{spectral-spanning-tree-claim9}(ii), yields that $N_{\overline{G}[U_3]}(w_1)\neq \emptyset$.

For any $j\in \{1,2,\ldots,t_2\}$, let $d_{U_2}(w_j)=r_j$.
By Claim \ref{spectral-spanning-tree-claim8}(ii), we have for any $j\in \{1,2,\ldots,t_2\}$,
\begin{equation}\label{EDST-th1-eq80}
  \text{$N_{U_2}(w_j)=\{u_1,u_2,\ldots,u_{r_j}\} $ with $t_1=r_1\ge r_2\ge \cdots \ge r_{t_2}$.}
\end{equation}
Thus
\begin{equation}\label{EDST-th1-eq84}
 \bigcup_{j=2}^{t_2-1}\{w_{j}u_i : j\le i\le r_j\}\subseteq E(G).
\end{equation}
By \eqref{EDST-th1-eq80}, Claims  \ref{spectral-spanning-tree-claim7}(ii) and \ref{spectral-spanning-tree-claim9}(ii), we have
$N_{\overline{G}[U_3]}(w_j)=\{w_1,\ldots,w_{j-1},w_{j+1},
\ldots,w_{r_j+1}\}$ for any $j\in \{1,2,\ldots,t_2\}$. 
Combining this with Claim \ref{spectral-spanning-tree-claim9}(ii), we get for any $j\in \{1,2,\ldots,t_2\}$,
\[
\{w_{j+1},\ldots,w_{r_j+1}\}\subseteq  N_{\overline{G}[U_3]}(w_j)\subseteq N_{\overline{G}[U_3]}[w_{j-1}]\subseteq \cdots \subseteq  N_{\overline{G}[U_3]}[w_{1}].
\]
Thus for any $w\in \{w_{j+1},\ldots,w_{r_j+1}\}$, we have $d_{\overline{G}[U_3]}(w)\ge j$. This, together with Claim \ref{spectral-spanning-tree-claim7}(v) and $d_{\overline{G}[U_3]}(w)=d_{U_2}(w)$ in Claim \ref{spectral-spanning-tree-claim7}(ii), yields that
\[
N_{\overline{G}[U_3]}[w_j]=\{w_1,\ldots,w_{j-1},w_{j},w_{j+1},
\ldots,w_{r_j+1}\}  \subseteq N_{U_3}(u_j).
\]
As $t_1=r_1\ge r_2\ge \cdots \ge r_{t_2}$, we have for any $j\in \{1,2,\ldots,t_2\}$,
\begin{equation*}
N_{\overline{G}[U_3]}[w_j]=\{w_1,\ldots,w_{j-1},w_{j},w_{j+1},
\ldots,w_{r_j+1}\} = N_{U_3}(u_j).
\end{equation*}
Therefore,
\begin{equation}\label{EDST-th1-eq85}
 \bigcup_{j=2}^{t_2-1}\{w_{i+1}u_j : j\le i\le r_j\}\subseteq E(G).
\end{equation}
Let  $W_2^j=N_{\overline{G}[U_3]}(w_j)\setminus \{w_1,\ldots,w_{j-1}\}$ for any $j\in \{1,2,\ldots,t_2-1\}$. Then
\begin{equation*}
  \text{$W_2^j=\{w_{j+1},\ldots,w_{r_j+1}\}\subseteq N_{\overline{G}[U_3]}(w_j) \subseteq W_3$  and
$|W_2^j|=r_j-j+1$ for any $j\in \{1,2,\ldots,t_2-1\}$.}
\end{equation*} 
Thus for each $j\in\{1,2,\ldots,t_2-1\}$, we have
\begin{equation}\label{EDST-th1-eq81}
\{w_{i+1}w_j : j\le i\le r_j\}
\subseteq E(\overline{G}).
\end{equation}
Therefore,
\begin{equation}\label{EDST-th1-eq86}
\overline{G}[W_3]-\bigcup_{j=2}^{t_2-1}\{w_{i+1}w_j : j\le i\le r_j\}=K_{1,t_2-1},
\end{equation}
where $K_{1,t_2-1}$ is a star with center $w_1$
and leaves $w_2,w_3,\ldots,w_{t_2}$.

By Claim \ref{spectral-spanning-tree-claim9}(ii), we have
$\sum_{1\le j\le t_2-1}|W_2^j|= e(\overline{G}[U_3])$.
As $e(\overline{G}[U_3])=s-k$ by Claim \ref{spectral-spanning-tree-claim7}(i), we have
\begin{equation}\label{spec-span-tree-eq1}
\sum_{2\le j\le t_2-1}|W_2^j|= e(\overline{G}[U_3])-|W_2^1|=s-k-t_1<s-k.
\end{equation}
By Claims \ref{spectral-spanning-tree-claim3} and \ref{spectral-spanning-tree-claim4}, we have
$|U_3|=p-1=2s-2k+2$.
Combining this with (\ref{spectral-spanning-tree-eq27}), we have
\begin{equation}\label{spec-span-tree-eq2}
|N_{U_3}(w_1)|=|U_3|-|W_3|=
2s-2k+2-t_2\ge s-k+2.
\end{equation}
By Claim \ref{spectral-spanning-tree-claim9}(ii), we have $N_{U_3}(w_{i})\subseteq
N_{U_3}[w_{i+1}]$,  which implies that every vertex of $N_{U_3}(w_{i})$ is adjacent to $w_{i+1}$ for any $i\in \{1,2,\ldots,t_2-1\}$. Combining this with (\ref{spec-span-tree-eq1}) and (\ref{spec-span-tree-eq2}),   for every $j\in \{2,3,\ldots,t_2-1\}$, we can choose vertex set
\begin{equation}\label{EDST-th1-eq87}
  \text{$V^j=\{v_j^j,v_{j+1}^j,\dots,v_{r_j}^j\}\subseteq N_{U_3}(w_1)\setminus W_3=N_{U_3}(w_1) \subseteq W_4$ with  $|V^j|=|W_2^j|, $}
\end{equation}
such that the sets $V^2,V^3,\ldots,V^{t_2-1}$ are pairwise disjoint. 
By \eqref{EDST-th1-eq87} and $N_G(u_1)\cap W_4 =\emptyset$, we have
\begin{equation}\label{EDST-th1-eq88}
\bigcup_{j=2}^{t_2-1}\{w_1v_i^j: j\le i\le r_j\}\subseteq E(G) \text{~and~} \bigcup_{j=2}^{t_2-1}\{u_1v_i^j: j\le i\le r_j\}\subseteq E(\overline{G}).
\end{equation}

As $n\ge 3s^2$ and $5\le k+3\le s\le 2k-2$,
by (\ref{spectral-spanning-tree-eq27}), we have
$ |N_{\overline{G}[V_p]}(w_1)|= |U_2|- |W_2|=n-s-1-t_1> n-s-1-(s-k)\ge 3s^2-2s+k-1=3s(s-1)+2k-1+s-k >s-k$. Combining this with (\ref{spec-span-tree-eq1}) and Claim \ref{spectral-spanning-tree-claim8}(ii),   for every $j\in \{2,\ldots,t_2-1\}$,  we can choose vertex set
\begin{equation*}
\text{ $Z^j=\{z_j^j,z_{j+1}^j,\dots,z_{r_j}^j\}\subseteq N_{\overline{G}[V_p]}(w_1)\subseteq N_{\overline{G}[V_p]}(w_j)\subseteq W_5$ with  $|Z^j|=|W_2^j|$,}
\end{equation*}
such that the sets  $Z^2,Z^3,\ldots,Z^{t_2-1}$ are pairwise disjoint. Thus \begin{equation}\label{EDST-th1-eq89}
\bigcup_{j=2}^{t_2-1}\{w_1z_i^j: j\le i\le r_j\}\subseteq  E(\overline{G}).
\end{equation}

Let
\begin{align}\label{spec-EDST-eq8}
G_2=G+\sum_{2\le j\le t_2-1}\Big(\sum_{j\le i\le r_j}(z_i^jw_1-w_1v_i^j+v_i^ju_1-u_iw_j+w_{i+1}w_j-w_{i+1}u_j)\Big).
\end{align}
Then we have $G_2=H_1(s,k)\in \mathcal{H}_1(s,k)$ and $e(G_2)=e(G)$ by \eqref{EDST-th1-eq84}, \eqref{EDST-th1-eq85}, \eqref{EDST-th1-eq81}, \eqref{EDST-th1-eq88} and \eqref{EDST-th1-eq89}.  Combining these with Lemma \ref{EDST-connect-le2} and Claim \ref{spectral-spanning-tree-claim7}(iii), we have
\begin{align}\label{EDST-eigenvalue-le1}
  n-2s+2k-3<\lambda(G_2)<n-2s+2k-2.
\end{align}

Let $W_6=(\cup_{2\le j\le t_2-1} Z^j) \cup (W_2\setminus\{u_1\})$,
$W_7=U_2\setminus (\{u_1\}\cup W_6)$, $W_8=(\cup_{2\le j\le t_2-1} V^j)\cup (W_3\setminus \{w_1\})$ and $W_9=U_3\setminus (W_8\cup \{w_1\})$. Then we have $|W_6|=s-k-1$, $|W_7|=n-2s-1+k$, $|W_8|=s-k$ and $|W_9|=s-k+1$.

Let $\boldsymbol{\eta}$ be the Perron vector of $A(G_2)$ and  $\eta_v$ be the coordinate in $\boldsymbol{\eta}$  corresponding to  vertex $v\in V(G_2)$.
For short, write $\lambda_2$ for $\lambda(G_2)$. Let $u'\in U_1$, $u_0\in W_7$ and $v_0\in W_9$.
As $G_2=H_1(s,k)$, by $A(G_2)\boldsymbol{\eta}=\lambda_2\boldsymbol{\eta}$, we
have
\begin{align*}
& \eta_{u}=\eta_{u'} \text{~for any~} u\in U_1,
\\
& \eta_{u}=\eta_{u_2}  \text{~for any~} u\in W_6,
\\
& \eta_{u}=\eta_{u_0}  \text{~for any~} u\in W_7,
\\
& \eta_{w}=\eta_{w_2} \text{~for any~} w\in W_8,
\\
 & \eta_{w}=\eta_{v_0}   \text{~for any~}  w\in W_9.
\end{align*}
Moreover, we have
\begin{align}
& \lambda_2\eta_{u_1}=|U_1|\eta_{u'}+(s-k-1)\eta_{u_2}+(n-2s-1+k)\eta_{u_0}
                       +\eta_{w_1}+(s-k)\eta_{w_2},
\label{spec-span-tree-eq3}\\
& \lambda_2\eta_{u_2}=|U_1|\eta_{u'}+\eta_{u_1}+(s-k-2)\eta_{u_2}
                       +(n-2s-1+k)\eta_{u_0} +\eta_{w_1},
\label{spec-span-tree-eq4}\\
& \lambda_2\eta_{u_0}=|U_1|\eta_{u'}+\eta_{u_1}+(n-2s-2+k)\eta_{u_0}
                       +(s-k-1)\eta_{u_2},
\label{spec-span-tree-eq5}\\
& \lambda_2\eta_{w_1}=|U_1|\eta_{u'}+\eta_{u_1}+(s-k-1)\eta_{u_2}
                       +(s-k+1)\eta_{v_0},
\label{spec-span-tree-eq6}\\
&\lambda_2\eta_{w_2}=|U_1|\eta_{u'}+\eta_{u_1}+(s-k-1)\eta_{w_2}
                      +(s-k+1)\eta_{v_0},
\label{spec-span-tree-eq7}\\
&\lambda_2\eta_{v_0}=|U_1|\eta_{u'}+\eta_{w_1}+(s-k)\eta_{w_2}
                      +(s-k)\eta_{v_0}.
\label{spec-span-tree-eq8}
\end{align}

Let $\eta_{z_1}=\max\{\eta_{w_1},\eta_{w_2},\eta_{v_0}\}$. By \eqref{spec-span-tree-eq6}--\eqref{spec-span-tree-eq8}, we have
\begin{align*}
\lambda_2\eta_{z_1}\le |U_1|\eta_{u'}+\eta_{u_1}+(s-k-1)\eta_{u_2}
                       +(s-k+1)\eta_{v_0} +(s-k)\eta_{w_2}+\eta_{w_1},
\end{align*}
Combining this with \eqref{spec-span-tree-eq5}, we have
\begin{align*}
\lambda_2\eta_{u_0}-\lambda_2\eta_{z_1}
\ge & (n-2s-2+k)\eta_{u_0}- (s-k+1)\eta_{v_0} -(s-k)\eta_{w_2}-\eta_{w_1}
\\
\ge & (n-2s-2+k)\eta_{u_0}- 2(s-k+1)\eta_{z_1}.
\end{align*}
Thus  by $n\ge 3s^2$, $k+3\le s\le 2k-2$  and $n-2s+2k-3<\lambda_2$, we have
\begin{align*}
\eta_{u_0} \ge \frac{\lambda_2-2s+2k-2}{\lambda_2-n+2s+2-k}\eta_{z_1}>\eta_{z_1}.
\end{align*}
This, together with (\ref{spec-span-tree-eq3})--(\ref{spec-span-tree-eq8}), we have \begin{align}\label{spec-span-tree-eq0}
\eta_{u_1}>\eta_{u_2}>\eta_{u_0}>\eta_{z_1}=\eta_{w_1}>\eta_{w_2}>\eta_{v_0}.
\end{align}
By  (\ref{spec-span-tree-eq3}) and (\ref{spec-span-tree-eq4}), we have
\begin{align}\label{spec-span-tree-eq9}
\eta_{u_1} = \eta_{u_2}+\frac{s-k}{\lambda_2+1}\eta_{w_2}.
\end{align}
By (\ref{spec-span-tree-eq6}) and (\ref{spec-span-tree-eq7}), we have
\begin{align}\label{spec-span-tree-eq10}
\eta_{w_1} = \eta_{w_2}+\frac{s-k-1}{\lambda_2}(\eta_{u_2}-\eta_{w_2}).
\end{align}
By (\ref{spec-span-tree-eq5}), (\ref{spec-span-tree-eq7})  and  (\ref{spec-span-tree-eq0}), we have
\begin{align*}
\lambda_2\eta_{u_0}-\lambda_2\eta_{w_2}
=&(n-2s-2+k)\eta_{u_0} +(s-k-1)\eta_{u_2}
  -(s-k-1)\eta_{w_2} -(s-k+1)\eta_{v_0}
 \\
\ge & (n-s-3)\eta_{u_0}-(2s-2k)\eta_{w_2}.
\end{align*}
Combining this with $n\ge 3s^2$, $k+3\le s\le 2k-2$ and $n-2s+2k-3<\lambda_2<n-2s+2k-2$, we get that
\begin{align}\label{spec-span-tree-eq11}
\eta_{u_0}\ge \frac{\lambda_2-2s+2k}{\lambda_2-n+s+3}\eta_{w_2}
\ge \frac{n-4s+4k-3}{2k-s+1}\eta_{w_2}>2\eta_{w_2}.
\end{align}

As $n-2s+2k-3<\lambda<n-2s+2k-2$ and $n-2s+2k-3<\lambda_2<n-2s+2k-2$, we have
$\lambda+1>\lambda_2$. Combining this with (\ref{spectral-spanning-tree-eq18-2}), (\ref{spectral-spanning-tree-eq18-0}), (\ref{spectral-spanning-tree-eq27}), (\ref{spec-span-tree-eq9}), and (\ref{spec-span-tree-eq10}),   we have
\begin{align} \label{spectral-spanning-tree-eq66}
 &\eta_{z_i^j}x_{w_1}+\eta_{w_1}x_{z_i^j}
-\eta_{w_1}x_{v_i^j}-\eta_{v_i^j}x_{w_1}+\eta_{v_i^j}x_{u_1}+\eta_{u_1}x_{v_i^j}
-\eta_{u_{i}}x_{w_j}-\eta_{w_j}x_{u_{i}}+\eta_{w_j}x_{w_{i+1}}+
\nonumber\\
&~~ \eta_{w_{i+1}}x_{w_j}
  -\eta_{w_{i+1}}x_{u_j}-\eta_{u_j}x_{w_{i+1}}
\nonumber
\\ 
 =&
 \eta_{u_2}x_{w_1}+\eta_{w_1}x_{z_i^j}
-\eta_{w_1}x_{v_i^j}-\eta_{w_2}x_{w_1}+\eta_{w_2}x_{u_1}+\eta_{u_1}x_{v_i^j}
-\eta_{u_2}x_{w_j}-\eta_{w_2}x_{u_{i}}+\eta_{w_2}x_{w_{i+1}}+
\nonumber\\
& ~~\eta_{w_2}x_{w_j}
  -\eta_{w_2}x_{u_j}-\eta_{u_2}x_{w_{i+1}}
\nonumber
\\
=&
(x_{w_1}-x_{w_{i+1}})(\eta_{u_2}-\eta_{w_2})+
\eta_{w_2}(x_{u_1}-x_{u_{i}}-x_{u_j}+x_{w_j})+
 \eta_{w_1}(x_{z_i^j}-x_{v_i^j})
+\eta_{u_1}x_{v_i^j}-x_{w_j}\eta_{u_2}
\nonumber
\\
\ge &
\eta_{w_2}(-x_{u_j}+x_{w_j})+ \eta_{w_1}(x_{z_i^j}-x_{v_i^j})
+\eta_{u_1}x_{v_i^j}-x_{w_j}\eta_{u_2}
\nonumber
\\
= & (
\eta_{w_2}(-x_{u_j}+x_{w_j})+  \bigl(\eta_{w_2}+\frac{s-k-1}{\lambda_2}(\eta_{u_2}-\eta_{w_2})\bigr)
(x_{z_i^j}-x_{v_i^j})
 +(\eta_{u_2}+\frac{s-k}{\lambda_2+1}\eta_{w_2})x_{v_i^j}-x_{w_j}\eta_{u_2}
\nonumber\\
= &
(\eta_{u_2}-\eta_{w_2})\bigl(\frac{s-k-1}{\lambda_2}(x_{z_i^j}-x_{v_i^j})
+x_{v_i^j}-x_{w_j}\bigr)+\eta_{w_2}(x_{z_i^j}-x_{u_j}+ \frac{s-k}{\lambda_2+1}x_{v_i^j})
\nonumber\\
\ge &
(\eta_{u_2}-\eta_{w_2})\bigl(\frac{s-k-1}{\lambda_2}(x_{z_0}-x_{w_0'})
+x_{v_i^j}-x_{w_2}\bigr)+\eta_{w_2}(x_{z_i^j}-x_{u_j})
\nonumber\\
\ge &
(\eta_{u_0}-\eta_{w_2})\bigl(\frac{s-k-1}{\lambda_2}(x_{z_0}-x_{w_0'})
+x_{v_i^j}-x_{w_2}\bigr)-\frac{t_2 }{\lambda + 1}x_{w_0}\eta_{w_2}
\nonumber\\
> &
(\eta_{u_0}-\eta_{w_2})\bigl(\frac{s-k-1}{\lambda_2}(x_{z_0}-x_{w_0'})
+x_{v_i^j}-x_{w_2}\bigr)-\frac{s-k }{\lambda_2}x_{w_0}\eta_{w_2}.
\end{align}

\noindent
{\bf Case 1. $s-k\ge 4$.}
As $n\ge 3s^2$ and $s\le 2k-2$,
by Claim \ref{spectral-spanning-tree-claim7}(iv), Claim \ref{spec-span-tree-claim10-1}(ii),  (\ref{spectral-spanning-tree-eq17-1}), (\ref{spectral-spanning-tree-eq27}), and   (\ref{EDST-eigenvalue-le1}), we have for any $w\in W_4$,
\begin{align}\label{spec-span-tree-eq12}
 & \frac{s-k-1}{\lambda_2}(x_{z_0}-x_{w_0'}) +x_{w}-x_{w_2}
\nonumber\\
\ge & \frac{s-k-1}{\lambda_2}(x_{z_0}-x_{w_0}) +x_{w}-x_{w_2}
 \nonumber\\
\ge & \frac{s-k-1}{\lambda_2}(x_{z_0}-x_{w_0}) + \frac{1}{\lambda + 1}x_{w_0} -\frac{s-k+1}{2(\lambda + 1)}x_{z_0}-\frac{t_2(s-k+1)}{2(\lambda + 1)^2}x_{w_0}
\nonumber\\
= & (\frac{s-k-1}{\lambda_2}-\frac{s-k+1}{2(\lambda + 1)})x_{z_0}+
(\frac{1}{\lambda + 1}-\frac{s-k-1}{\lambda_2} -\frac{t_2(s-k+1)}{2(\lambda + 1)^2})x_{w_0}
\nonumber\\
\ge & (\frac{s-k-1}{\lambda_2}-\frac{s-k+1}{2\lambda_2})x_{z_0}+
(\frac{1}{\lambda + 1}-\frac{s-k-1}{\lambda_2} -\frac{(s-k)(s-k+1)}{2(\lambda + 1)^2})x_{w_0}
\nonumber\\
\ge & \frac{s-k-3}{2\lambda_2}x_{z_0}+
(\frac{1}{\lambda + 1}-\frac{s-k-1}{\lambda_2} -\frac{(k-2)(k-1)}{2(\lambda+1)^2})x_{w_0}
\nonumber\\
> & \frac{s-k-3}{2\lambda_2}x_{z_0} -\frac{s-k-1}{\lambda_2} x_{w_0}
\nonumber\\
\ge & (\frac{(s-k-3)(n-4s+4k-3)}{2s\lambda_2}-\frac{s-k-1}{\lambda_2} )x_{w_0}
\nonumber\\
\ge & (\frac{3s^2-4k+5}{2s\lambda_2}-\frac{s-k-1}{\lambda_2}) x_{w_0}
\nonumber\\
> & (\frac{s}{\lambda_2}-\frac{s-k-1}{\lambda_2}) x_{w_0}
\ge \frac{s-k}{\lambda_2} x_{w_0}.
\end{align}
Combining this with \eqref{spec-EDST-eq8}, \eqref{spectral-spanning-tree-eq66},  and $\eta_{u_0}>2\eta_{w_2}$ by \eqref{spec-span-tree-eq11},  we have $${\boldsymbol{\eta}}^T (\lambda(G_2)-\lambda(G)) {\mathbf x}>\frac{s-k }{\lambda_2}x_{w_0}\eta_{w_2}-\frac{s-k }{\lambda_2}x_{w_0}\eta_{w_2}=0,$$
and thus  $\lambda(G_2)>\lambda(G)$.
This, together with $G_2$ is $s$-edge-connected by Lemma \ref{EDST-connect-le1}, yields that $t_1<s-k$ is impossible. Thus we have $t_1=s-k$ in this case.

\noindent
{\bf Case 2. $s-k= 3$.}  As $n\ge 3s^2$ and $k+3=s\le 2k-2$,
by Claim \ref{spectral-spanning-tree-claim7}(iv), Claim \ref{spec-span-tree-claim10-1}(ii),  (\ref{spectral-spanning-tree-eq17-1}), (\ref{spectral-spanning-tree-eq18-0}) and  (\ref{EDST-eigenvalue-le1}),    we have for any $w\in W_4$,
\begin{align}\label{spec-span-tree-eq12}
 & \frac{s-k-1}{\lambda_2}(x_{z_0}-x_{w_0'}) +x_{w}-x_{w_2}
\nonumber\\
 =&
  \frac{2}{\lambda_2}(x_{z_0}-x_{w})+x_{w}-x_{w_0}
 \nonumber\\
 =&
     \frac{2}{\lambda_2}x_{z_0} +\frac{\lambda_2-2}{\lambda_2}x_{w}
      -x_{w_0}
 \nonumber\\
\ge &
       \frac{2}{\lambda_2}x_{z_0} +\frac{\lambda_2-2}{\lambda_2}
       \Big(x_{w_0}+ \frac{2}{\lambda + 1}x_{w_0} -\frac{2}{\lambda + 1}x_{z_0}-\frac{6}{(\lambda + 1)^2}x_{w_0} \Big)
       -x_{w_0}
\nonumber\\
\ge &
       \frac{2(\lambda -\lambda_2+ 3)}{\lambda_2(\lambda + 1)}x_{z_0} +\frac{\lambda_2-2}{\lambda_2}
       \Big( \frac{2}{\lambda + 1}x_{w_0}-\frac{6}{(\lambda + 1)^2}x_{w_0} \Big)
       -\frac{2}{\lambda_2}x_{w_0}
\nonumber\\
\ge &
       \frac{4}{\lambda_2(\lambda + 1)}x_{z_0} +\frac{\lambda_2-2}{\lambda_2}
       \Big( \frac{2}{\lambda + 1} -\frac{6}{(\lambda + 1)^2}\Big)x_{w_0}
       -\frac{2}{\lambda_2}x_{w_0}
\nonumber\\
= &
       \frac{4}{\lambda_2(\lambda + 1)}x_{z_0} +
      \frac{2(\lambda_2+2)-8}{\lambda_2(\lambda + 1)}x_{w_0} -\frac{6(\lambda_2-2)}{\lambda_2(\lambda + 1)^2}x_{w_0}
       -\frac{2}{\lambda_2}x_{w_0}
\nonumber\\
\ge &
       \frac{4}{\lambda_2(\lambda + 1)}x_{z_0} -
      \frac{8}{\lambda_2(\lambda + 1)}x_{w_0} -\frac{6(\lambda_2-2)}{\lambda_2(\lambda + 1)^2}x_{w_0}
\nonumber\\
      \ge &
       \frac{4(n-4s+4k-3)}{s\lambda_2(\lambda + 1)}x_{w_0} -
      \frac{14}{\lambda_2(\lambda + 1)}x_{w_0}
\nonumber\\
\ge &
       \frac{4(3s^2-4k+5)}{s\lambda_2(\lambda + 1)}x_{w_0} -
      \frac{14}{\lambda_2(\lambda + 1)}x_{w_0}
\nonumber\\
\ge &
       \frac{8s}{\lambda_2(\lambda + 1)}x_{w_0} -
      \frac{14}{\lambda_2(\lambda + 1)}x_{w_0}
\nonumber\\
\ge &
       \frac{4s}{\lambda_2(\lambda + 1)}x_{w_0},
\end{align}
Recall that $\lambda<n-2s+2k-2$ and $k+3= s\le 2k-2$, we have $4(n-4s+4k-3)>3(\lambda+1)$.
Combining this with (\ref{spec-span-tree-eq11}) and \eqref{spec-span-tree-eq12}, we have
for any $w\in W_4$,
\begin{align*}
&(\eta_{u_0}-\eta_{w_2})\Big(\frac{s-k-1}{\lambda_2}(x_{z_0}-x_{w_0'}) +x_{w}-x_{w_2}\Big)
\\
\ge& \frac{4s(n-3s+2k-4)}{\lambda_2(2k-s+1)(\lambda + 1)}x_{w_0}\eta_{w_2}
\\
>& \frac{4(n-3s+2k-4)}{\lambda_2(\lambda + 1)}x_{w_0}\eta_{w_2}\\
>&\frac{3}{\lambda_2} x_{w_0}\eta_{w_2}= \frac{s-k}{\lambda_2} x_{w_0}\eta_{w_2}.
\end{align*}
Then by \textcolor{blue}{ \eqref{spec-EDST-eq8} and \eqref{spectral-spanning-tree-eq66}, }    we have $\lambda(G_2)>\lambda(G)$.
This, together with $G_2$ is $s$-edge-connected by Lemma \ref{EDST-connect-le1}, yields that $t_1<s-k$ is impossible.
Thus we have $t_1=s-k$, and so Claim \ref{spectral-spanning-tree-claim10} holds.
\q
\par\medskip
\noindent{\bf Proof of Theorem \ref{spectral-spanning-tree-thm1}:}
By \eqref{spectral-spanning-tree-eq19}, (\ref{spectral-spanning-tree-eq27-0}), Claims \ref{spectral-spanning-tree-claim8}, \ref{spectral-spanning-tree-claim7} and \ref{spectral-spanning-tree-claim10},  we have $G=H_1(s,k)$. As
$G$ is the graph with the maximum spectral radius  among all $s$-edge-connected graphs satisfying $\tau(G)\le k-1$, we get Theorem \ref{spectral-spanning-tree-thm1} holds.
\q

\subsection{Proof of Theorem \ref{spectral-spanning-tree-thm2} }

As $K_{n-2k-1}\varsubsetneqq H_2(k)$, we have $\lambda(G)>\lambda (K_{n-2k-1})= n-2k-2$.
Combining this with Lemma \ref{le3-1}, we get
\[
n-2k-2<\lambda(G)\le
\frac{2k-2+\sqrt{4k^2+4(2e(G)-2kn+n)}}{2},
\]
which implies
\begin{align}\label{specST-th2-eq0}
e(G)>
\frac{(2n-6k-2)^2-4k^2+8kn-4n}{8}
=
\frac{1}{2} n^2-2kn-\frac{3}{2}n +4k^2+3k+\frac{1}{2}.
\end{align}
Combining \eqref{spectral-spanning-tree-eq-add1} with (\ref{specST-th2-eq0}), we have
\[
\binom{n-p+1}{2}+k(p-1)-1
>
\frac{1}{2} n^2-2kn-\frac{3}{2}n +4k^2+3k+\frac{1}{2},
\]
which is equivalent to
\begin{align}\label{specST-th2-eq3-0}
f(p)>0, \text{ where } f(p)=p^2+(2k-2n-1)p+4kn+4n-8k^2-8k-3.
\end{align}


\begin{Claim}\label{specST-th2-claim2}
 $p=|\pi|=2k+2$.
\end{Claim}
\p
Suppose, to the contrary, that $p \neq 2k+2$. Then by Claim \ref{spectral-spanning-tree-claim1}(i), we have $p \ge 2k+3$. Combining this with
(\ref{specST-th2-eq3-0}),  we obtain that for $p\in [2k+3,n]$, the maximum value  of $f(p)$ is attained at an endpoint. Thus $f(p)\le \max\{f(2k+3),\,f(n)\}$, where
\begin{align*}
f(2k+3) = 8k-2n+3 \text{ and }
f(n) = -n^2 + 6kn + 3n - 8k^2 - 8k - 3.
\end{align*}
By  $n\ge 8k^2-2k$, we have
$f(p)\le \max\{f(2k+3),\,f(n)\}<0$, contrary to (\ref{specST-th2-eq3-0}). So Claim \ref{specST-th2-claim2} holds.\q


\begin{Claim}\label{specST-th2-claim3}
 $t=p-1=2k+1$.
\end{Claim}
\p
On the contrary, suppose that $t\neq p-1$. Then $t=p=n$ or $0\le t\le p-2=2k$.
As $n\ge 8k^2-2k$ and $p=2k+2$ in Claim \ref{specST-th2-claim2}, we have $t=p=n$ is impossible. Thus we have $0\le t\le p-2=2k$.
By Claim \ref{spectral-spanning-tree-claim1}(ii),
Claim \ref{specST-th2-claim2}   and Corollary \ref{21c}, we have
\begin{align*}
e(G)\le& \binom{n-(p-t-1)(k+2)-t}{2}+(p-t-1)\binom{k+2}{2}+k(p-1)-1,
\\
\le& \binom{n-k-p}{2}+\binom{k+2}{2}+k(p-1)-1
\\
=& \frac{n^2}{2} - 3nk + 7k^2 - \frac{5}{2}n + 10k + 3.
\end{align*}
Combining this with (\ref{specST-th2-eq0}), we have
\begin{align}\label{specST-th2-eq4}
h(n) < 0, \text{ where } h(n)=nk + n - 3k^2 - 7k - \frac{5}{2}.
\end{align}
Since $n\ge 8k^2-2k$ and $k\ge 2$, it follows that  $h(n)\ge 0$,
contrary to (\ref{specST-th2-eq4}).
So Claim \ref{specST-th2-claim3} holds.
\q

Let $U_3=\cup_{i=1}^{p-1}V_i$ and $W=\{w_1,w_2,\ldots,w_{t_0}\}$ be a subset of $U_3$  satisfies that $d_{\overline{G}[U_3]}(w_i)\ge 1$ for any $i\in \{1,2,\ldots,t_0\}$ and $d_{V_p}(w_1)\ge d_{V_p}(w_2)\ge \cdots \ge d_{V_p}(w_{t_0})$. Recall that ${\bf x}$ is the corresponding Perron vector of $\lambda(G)$. If there exists $i\in \{1,2,\ldots,t_0\}$ such that $d_{V_p}(w_i)= d_{V_p}(w_{i+1})$, then we can further  define   $W$  satisfies that $x_{w_{i}}\ge x_{w_{i+1}}$.

\begin{Claim}\label{specST-th2-claim5}
Each of the following holds:
\\
(i) $G[V_p]$ is a complete graph with  $|V_p|=n-p+1=n-2k-1$;
\\
(ii) $e(G[U_3])=2k^2-k$, $e(\overline{G}[U_3])=2k$, $e(V_p, U_3)=2k-1$  and $d_G(v)=2k-1$ for any $v\in U_3$;
\\
(iii)  $d_{\overline{G}[U_3]}(v)=d_{V_p}(v)+1$ for any $v\in U_3$;
\\
(iv) $U_3=W$ and $t_0=|U_3|=2k+1=e(\overline{G}[U_3])+1$;
\\
(v) $2\le \triangle({\overline{G}[U_3]})\le k$.
\end{Claim}
\p
As $G$ is a graph with the maximum spectral radius  among all $(2k-1)$-edge-connected graphs satisfying $\tau(G)\le k-1$, by the choice of $\pi$, we have $G[V_i]$ is a complete graph for any $i\in \{1,2,\ldots,p\}$. Combining this with Claims \ref{specST-th2-claim2} and \ref{specST-th2-claim3}, we have $G[V_p]$ is a complete graph with  $|V_p|=n-p+1=n-2k-1$. This proves Claim \ref{specST-th2-claim5}(i).

By Claim \ref{specST-th2-claim3}, we have $|U_3|=p-1=2k+1$.
Since
\begin{align*}
\frac{\sum_{v\in U_3}d_G(v)-e(V_p, U_3)}{2}=
e(G[U_3])=e_G(\pi)-e(V_p, U_3),
\end{align*}
it follows that
\begin{align*}
\sum_{v\in U_3}d_G(v) = 2e_G(\pi)-e(V_p, U_3).
\end{align*}
This, together with $G$ being  $(2k-1)$-edge-connected, (\ref{spectral-spanning-tree-eq1}),  Claims \ref{specST-th2-claim2} and \ref{specST-th2-claim3}, yields that
\begin{align*}
(2k+1)(2k-1)\le \sum_{v\in U_3}d_G(v) = 2e_G(\pi)-e(V_p, U_3)\le 2\big[k(2k+1)-1\big]-(2k-1).
\end{align*}
By the above expression,   we have $\sum_{v\in U_3}d_G(v)=4k^2-1$, $e_G(\pi)=2k^2+k-1$ and $e(V_p, U_3)=2k-1$. Thus $d_G(v)=2k-1$ for any $v\in U_3$ and $2e(G[U_3])=\sum_{v\in U_3}d_G(v)-e(V_p, U_3)=2k(2k-1).$ Furthermore, we have
 $$
 e(\overline{G}[U_3])=
 \binom{|U_3|}{2}-e(G[U_3])
 =\binom{2k+1}{2}-(2k^2-k)=2k.
 $$
 This proves Claim \ref{specST-th2-claim5}(ii).

As $d_G(v)=2k-1$ for any $v\in U_3$, we have
 $d_{U_3}(v)+d_{V_p}(v)=2k-1$.
 Since  $|U_3|=2k+1$, it follows that
 $d_{U_3}(v)+d_{\overline{G}[U_3]}(v)=2k$.
 Thus we have
 \[d_{\overline{G}[U_3]}(v)=
 2k-d_{U_3}(v)=d_{V_p}(v)+1.
 \]
This proves Claim \ref{specST-th2-claim5}(iii).

By Claim \ref{specST-th2-claim5}(ii) and
the definition of $W$, we have  $U_3=W$.
As $|U_3|=|\cup_{i=1}^{p-1}V_i|=p-1$, by Claims \ref{specST-th2-claim3} and   \ref{specST-th2-claim5}(ii), we have $t_0=|W|=|U_3|=2k+1=e(\overline{G}[U_3])+1$.
This proves Claim \ref{specST-th2-claim5} (iv).

As $e(\overline{G}[U_3])=2k$ in Claim \ref{specST-th2-claim5}(ii) and $|U_3|=e(\overline{G}[U_3])+1$ in Claim \ref{specST-th2-claim5}(iv), we get $\triangle({\overline{G}[U_3]})\ge 2$. We next prove $\triangle({\overline{G}[U_3]})\le k$ by contradiction.
Suppose that $\triangle({\overline{G}[U_3]})\ge k+1$.  Then there exists a vertex $v\in U_3$ such that $d_{\overline{G}[U_3]}(v)\ge k+1$.
By Claim \ref{specST-th2-claim5}(ii)-(iii), we have \begin{align}\label{specST-th2-eq19}
  & e(V_p\cup \{v\},U_3\setminus\{v\})
\nonumber\\
= & e(V_p,U_3)-e(V_p, \{v\})+e(\{v\},U_3\setminus \{v\})
\nonumber\\
= & 2k-1-d_{V_p}(v)+d_{U_3}(v)
\nonumber\\
= & 2k-1-d_{V_p}(v)+(d_{G}(v)-d_{V_p}(v))
\nonumber\\
= & 4k-2-2(d_{\overline{G}[U_3]}(v)-1)
\le  2k-2,
\end{align}
contrary to $\kappa'(G)=2k-1$.  This proves Claim \ref{specST-th2-claim5}(v), and so Claim \ref{specST-th2-claim5} holds.
\q


As $A(G)\mathbf{x}=\lambda(G)\mathbf{x}$,  by Claim \ref{specST-th2-claim5}(i) and (ii),  for any $u\in V_p$ and $v\in U_3$, we have
\begin{align}
\lambda(G)x_{u}&=\sum_{w\in V_p\setminus \{u\}}x_w + \sum_{w\in N_G(u)\cap U_3 }x_w,
\label{specST-th2-eq13}
\\
\lambda(G)x_{v}&=\sum_{w\in N_G(v)\cap U_3 }x_w+\sum_{w\in N_G(v)\cap V_p} x_w. \label{specST-th2-eq15}
\end{align}

For short, we write $\lambda$ for $\lambda(G)$.
Let $x_{u_0}=\min\{x_{w}: w\in V_p\}$ and $x_{v_0}=\max\{x_{w}: w\in U_3\}$. By (\ref{specST-th2-eq13}), (\ref{specST-th2-eq15})  and Claim  \ref{specST-th2-claim5}(i)--(ii),  we have
\begin{align*}
\lambda x_{u_0}
&= \lambda x_{v_0}-\sum_{w\in N_G(v_0)\cap U_3 }x_w-\sum_{w\in N_G(v_0)\cap V_p}x_w
 +\sum_{w\in V_p\setminus \{u_0\}}x_w+\sum_{w\in N_G(u_0)\cap U_3 }x_w
\\
&\ge \lambda x_{v_0}-\sum_{w\in N_G(v_0)\cap U_3 }x_w-\sum_{w\in N_G(v_0)\cap V_p}x_w
 +\sum_{w\in V_p\setminus \{u_0\}}x_w
\nonumber \\
&\ge \lambda x_{v_0}-(2k-1)x_{v_0}
 +\sum_{w\in V_p\setminus N_G(v_0)}x_{w} -x_{u_0}
\nonumber \\
 &\ge (\lambda -2k+1)x_{v_0}+(n-2k-1-2k+1)x_{u_0}-x_{u_0},
\end{align*}
which implies that
\begin{align}\label{specST-th2-eq16}
(\lambda -n+4k+1)x_{u_0}\ge (\lambda -2k+1)x_{v_0}.
\end{align}
As $n\ge 8k^2-2k$, $n-1>\lambda > n-2k-2$ and $k\ge 2$, we have $\lambda -n+4k+1>0$.
Combining this with (\ref{specST-th2-eq16}), we have
\begin{align}\label{specST-th2-eq17}
x_{u_0}\ge \frac{\lambda -2k+1}{\lambda -n+4k+1} x_{v_0}> \frac{n-2k}{4k} x_{v_0} > x_{v_0},
\end{align}
where the second inequality holds as $(\lambda -2k+1)/(\lambda -n+4k+1)$ is decreasing on $\lambda$.

Let $U_1=\{u_1,u_2,\ldots,u_{t_1}\}$ satisfies that $d_{U_3}(u_i)\ge 1$ for any $i\in \{1,2,\ldots,t_1\}$ and $d_G(u_1)\ge d_G(u_2)\ge \cdots \ge d_G(u_{t_1})$, and  $U_2=V_p\setminus U_1$. Since $e(V_p, U_3)=2k-1$ by Claim \ref{specST-th2-claim5}(ii) and $n\ge 8k^2-2k$, we have $u_0\in U_2$ by \eqref{specST-th2-eq13}.

Let $W_1=\{w_1,w_2,\ldots,w_{t_2}\}\subseteq W=U_3$  satisfies that $d_{\overline{G}[U_3]}(w_i)\ge 2$ for any $i\in \{1,2,\ldots,t_2\}$.
As $W_1\subseteq W$,  we have $d_{V_p}(w_1)\ge d_{V_p}(w_2)\ge \cdots \ge d_{V_p}(w_{t_2})$. By Claim \ref{specST-th2-claim5}(iii), we have
$d_{V_p}(w_i)\ge 1$ for any $i\in \{1,2,\ldots,t_2\}$.
Let $W_2=U_3\setminus W_1$. By the definition of $W_1$ and Claim \ref{specST-th2-claim5}(ii), we get
for any $v\in U_3\setminus W_1$, $d_{U_3}(v)=d_G(v)=2k-1$.
This, together with $|U_3|=2k+1$  in Claim \ref{specST-th2-claim5}(iv), yields that every vertex of
$W_2$ is a leaf vertex in $\overline{G}[U_3]$, that is,
\begin{equation}\label{specST-th2-w2eq1}
\text{ for any $v\in W_2$, we have  $d_{\overline{G}[U_3]}(v)=1$. }
\end{equation}



\begin{Claim}\label{specST-th2-claim6}
Each of the following holds:
\\
(i) $N_{U_3}(u_1)\supseteq N_{U_3}(u_2)\supseteq \cdots \supseteq N_{U_3}(u_{t_1})$;
\\
(ii) $N_{V_p}(w_1)\supseteq N_{V_p}(w_2)\supseteq \cdots \supseteq N_{V_p}(w_{t_2})$, where $t_1=d_{V_p}(w_1)$ and $t_2=d_{U_3}(u_1)$;
\\
(iii)  if there exist $i,j\in \{1,2,\ldots,t\}$ with $i<j$ and $N_{V_p}(w_i)\supsetneqq N_{V_p}(w_{j})$, then $x_{w_{i}}>x_{w_{j}}$;
\\
(iv) $1\le t_1=d_{V_p}(w_1)=\triangle({\overline{G}[U_3]})-1\le k-1$.
\end{Claim}
\p
On the contrary, suppose that there exist $i,j\in \{1, 2, \ldots, t_1\}$ with $i<j$  such that $N_{U_3}(u_i)\nsupseteq N_{U_3}(u_j)$. As $G[V_p]$ is a complete graph and $d_G(u_i)\ge d_G(u_j)$, it follows that  $N_{U_3}(u_j)\setminus  N_{U_3}(u_i)\neq \emptyset$.
Then $\tau(G(u_i,u_j))\le k-1$ and $G(u_i,u_j)$ is $s$-edge-connected by Claim \ref{specST-th2-claim5}(v) and  Lemma \ref{EDST-thm2-le1}.
 Moreover, by Lemma \ref{le2.1} and $G(u_i,u_j)\ncong G$ from  $N_{U_3}(u_i)\setminus  N_{U_3}(u_j)\neq \emptyset$, we have $\lambda(G(u_i,u_j))>\lambda(G)$, contrary to the choice of $G$.
So Claim \ref{specST-th2-claim6}(i) holds.

We prove $N_{V_p}(w_1)\supseteq N_{V_p}(w_2)\supseteq \cdots \supseteq N_{V_p}(w_{t_2})$ by contradiction. Suppose that there exist $i, j \in \{1,2,\dots,t_2\}$ with $i<j$, such that $N_{V_p}(w_j) \nsubseteq N_{V_p}(w_{i})$. Then there exists $u_{j'}\in N_{V_p}(w_j) \setminus  N_{V_p}(w_{i})$. Since $d_{V_p}(w_i)\ge  d_{V_p}(w_j)$,   there exists $u_{i'}\in N_{V_p}(w_i) \setminus  N_{V_p}(w_{j})$.
If $j'<i'$, then by Claim \ref{specST-th2-claim6}(i), we have
$w_i\in N_{U_3}(u_{i'})\subseteq N_{U_3}(u_{j'})$, contrary to  $u_{j'}\in N_{V_p}(w_j) \setminus  N_{V_p}(w_i)$. Thus we must have $j'>i'$, again by Claim \ref{specST-th2-claim6}(i), we have
$w_j\in N_{U_3}(u_{j'})\subseteq N_{U_3}(u_{i'})$, contrary to  $u_{i'}\in N_{V_p}(w_i) \setminus  N_{V_p}(w_{j})$.  Therefore, $N_{V_p}(w_1)\supseteq N_{V_p}(w_2)\supseteq \cdots \supseteq N_{V_p}(w_{t_2})$.
 Combining this with Claim \ref{specST-th2-claim6}(i), we have $t_1=d_{V_p}(w_1)$ and $t_2=d_{U_3}(u_1)$. So Claim \ref{specST-th2-claim6}(ii) holds.

Next, we show (iii).
 As $N_{V_p}(w_i)\setminus N_{V_p}(w_j)\neq \emptyset$,
by (\ref{specST-th2-eq15}), (\ref{specST-th2-eq17}), $n\ge 8k^2-2k$ and Claim \ref{specST-th2-claim5}(ii), we have
 \begin{align*}
 \lambda  x_{w_i}- \lambda x_{w_{j}} &= \sum_{w\in N_G(w_i)\cap U_3 }x_w+\sum_{w\in N_G(w_i)\cap V_p} x_w -\sum_{w\in N_G(w_{j})\cap U_3 }x_w - \sum_{w\in N_G(w_{j})\cap V_p} x_w
 \\
 &\ge  \sum_{w\in N_G(w_i)\cap V_p} x_w - \sum_{w\in N_G(w_{j})\cap V_p} x_w-\sum_{w\in N_G(w_{j})\cap U_3 }x_w
 \\
 &= \sum_{w\in N_{V_p}(w_i)\setminus N_{V_p}(w_{j}) } x_w - \sum_{w\in N_G(w_{j})\cap U_3 }x_w
 \nonumber \\
 &\ge  x_{u_0}-(2k-1)x_{v_0}
 \nonumber \\
 &> \frac{n-2k}{4k} x_{v_0}-(2k-1)x_{v_0}
 = \frac{n-8k^2+2k}{4k} x_{v_0} \ge 0.
\end{align*}
Thus we have  $x_{w_i}>x_{w_{j}}$. This proves Claim \ref{specST-th2-claim6}(iii).

By \eqref{specST-th2-w2eq1}, Claims \ref{specST-th2-claim5}(iii)  and   \ref{specST-th2-claim6}(ii), we have $t_1=d_{V_p}(w_1)=\triangle({\overline{G}[U_3]})-1$. Thus by  Claim \ref{specST-th2-claim5}(v), we get
 $1\le t_1=d_{V_p}(w_1)=\triangle({\overline{G}[U_3]})-1\le k-1$.
This proves Claim \ref{specST-th2-claim6}(iv).
\q

By Claim \ref{specST-th2-claim5}(i)  and (\ref{specST-th2-eq13}), we have for any $u\in V_p\setminus U_1=U_2$,
\begin{align*}
\lambda x_{u} - \lambda x_{u_1}= x_{u_1} - x_{u} - \sum_{v\in N_G(u_1)\cap U_3}x_v.
\end{align*}
Combining this with $t_2=d_{U_3}(u_1)$ in  Claim \ref{specST-th2-claim6}(ii),
  we have
\begin{align}\label{specST-th2-eq66}
0> x_{u} - x_{u_1} \ge  -\frac{t_2 x_{v_0} }{\lambda + 1}.
\end{align}


\begin{Claim}\label{specST-th2-claim6-7}
Let $C_1, C_2, \ldots, C_r$ be the components of $\overline{G}[U_3]$. If $r\ge 2$, each of the following holds:
\\
(i) there exists a component  with at least one cycle;
\\
(ii)  $\Delta( C_i)\ge 2$ for any $i\in \{1,2,\ldots,r\}$;
\\
(iii) for any $w\in U_3\setminus \{w_1,w_2\}$,  $d_{\overline{G}[U_3]}(w)\le k-1$ when $k\ge 3$.
\end{Claim}
\p
By contradiction, suppose that all components are acyclic.
Then $|U_3|=|V(C_1)|+\cdots + |V(C_r)| \ge |e(C_1)|+\cdots + |e(C_r)|+r =e(\overline{G}[U_3])+r\ge e(\overline{G}[U_3])+2$. Combining this with Claim \ref{specST-th2-claim5}(ii), we have $|U_3|\ge 2k+2$, contrary to $|U_3|=2k+1$ by Claim \ref{specST-th2-claim3}.  So Claim \ref{specST-th2-claim6-7}(i) holds.

We will show (ii). On the contrary, suppose that there exists $i\in \{1,2,\ldots,r\}$ such that $\Delta( C_i)\le 1$. By Claim \ref{specST-th2-claim5}(iii), we have $\Delta( C_i)= 1$, and so $C_i=K_2$. By Claim \ref{specST-th2-claim6-7}(i), there exists $j\in \{1,2,\ldots,r\}$ with $j\neq i$ such that  $C_j$ has one cycle.
Let $V(C_i)=\{v_{i}',v_i''\}$ and $v_{j}'v_j''$ be an edge of a cycle in $C_j$.
Then $d_{\overline{G}[U_3]}(v_{i}')=d_{\overline{G}[U_3]}(v''_{i})=1$ and
$d_{\overline{G}[U_3]}(v_{j}')=d_{\overline{G}[U_3]}(v''_{j})\ge 2$.
By Claim \ref{specST-th2-claim5}(iii),  we have $d_{V_p}(v_{i}')=d_{V_p}(v''_{i})=0$ and
$d_{V_p}(v_{j}')=d_{V_p}(v''_{j})\ge 1$.
Thus by Claim \ref{specST-th2-claim6}(i) and (iii), we have $\min\{x_{v_{j}'},x_{v_{j}''}\}> \max\{x_{v_{i}'},x_{v_{i}''}\} $.
Let $H_1=G+v_{i}'v_{i}''-v_{i}'v_{j}' +v_{j}'v_{j}''-v_{j}''v_{i}''$.
Then  $\tau(H_1)\le k-1$ and  $H_1$  is $(2k-1)$-edge-connected by Lemma \ref{EDST-thm2-le1}. Moreover,
\begin{align*}
{\mathbf x}^T (A(H_1)-A(G)) {\mathbf x}
=2(x_{v_{i}'}x_{v_{i}''}-x_{v_{i}'}x_{v_{j}'}+
x_{v_{j}'}x_{v_{j}''}-x_{v_{j}''}x_{v_{i}''})
=2(x_{v_{i}'}-x_{v_{j}''})(x_{v_{i}''}-x_{v_{j}'})>0.
\end{align*}
Thus $\lambda(H_1)>\lambda(G)$, a contradiction.
So Claim \ref{specST-th2-claim6-7}(ii) holds.

On the contrary, suppose that there exists $w\in U_3\setminus \{w_1,w_2\}$ such that  $d_{\overline{G}[U_3]}(w)\ge k$.  By Claim \ref{specST-th2-claim6}(iv), we have $d_{\overline{G}[U_3]}(w)= k$. Combining this with Claim \ref{specST-th2-claim5}(iii), we get $d_{\overline{G}[U_3]}(w_1)=d_{\overline{G}[U_3]}(w_2)= k$. Thus $e(\overline{G}[U_3])=\frac{1}{2}\sum_{w\in U_3}  d_{\overline{G}[U_3]}(w)\ge    \frac{1}{2}(d_{\overline{G}[U_3]}(w_1)+d_{\overline{G}[U_3]}(w_2)
+d_{\overline{G}[U_3]}(w)+|U_3|-3)\ge \frac{1}{2}(5k-2)$, contrary to $k\ge 3$ and  $e(\overline{G}[U_3])=2k$ in Claim \ref{specST-th2-claim5}(ii). So Claim \ref{specST-th2-claim6-7}(iii) holds. This proves Claim \ref{specST-th2-claim6-7}.
\q

 \begin{Claim}\label{specST-th2-claim6v}
 If $k=2$, then $G\cong H_2(2)$.
 \end{Claim}
 \p
 We first prove that $\overline{G}[U_3]$ is connected. By contradiction, assume that $\overline{G}[U_3]$ is disconnected. By Claim \ref{specST-th2-claim6-7}(i)-(ii), we have $e(\overline{G}[U_3])\ge e(C_1)+\cdots+e(C_r)\ge 3+2=5$, contrary to  $e(\overline{G}[U_3])=2k=4$ in Claim \ref{specST-th2-claim5}(ii). Thus $\overline{G}[U_3]$ is connected. By Claim \ref{specST-th2-claim5}(iv), we get $\overline{G}[U_3]$ is tree, and so $\overline{G}[U_3]=P_5$. This proves Claim \ref{specST-th2-claim6v}.
 \q

 In what follows, denote by $W_3=\{v\in U_3: v\notin N_{\overline{G}[U_3]}(w_2) \text{~and~} N_{\overline{G}[U_3]}(v)\cap  (W_1\setminus \{w_1,w_2\})\ne \emptyset\}$ and  $W_4=N_{\overline{G}[U_3]}[w_1]\cup N_{\overline{G}[U_3]}[w_2]$.

\begin{Claim}\label{specST-th2-claim7-0}
If $k\ge 3$ and  $d_{\overline{G}[U_3]}(w_2)\le k-1$, then $W_3= \emptyset$.
\end{Claim}
\p
On the contrary, suppose that $W_3\neq \emptyset$. Let $w'\in W_3$. Then $w'w_2\in E(G)$ and  $N_{\overline{G}[U_3]}(w')\cap  (W_1\setminus \{w_1,w_2\})\neq \emptyset$. Thus there exits $w\in N_{\overline{G}[U_3]}(w')\cap  (W_1\setminus \{w_1,w_2\})$ such that $d_{\overline{G}[U_3]}(w)\ge 2$. Combining this with Claim \ref{specST-th2-claim5}(iii), we have  $d_{V_p}(w)=d_{\overline{G}[U_3]}(w)-1\ge 1$.  Therefore, by Claim \ref{specST-th2-claim6}(i), we have $wu_1\in E(G)$.

As $d_{\overline{G}[U_3]}(w_2)\le  k-1$, by Claim \ref{specST-th2-claim5}(iii), we have  $d_{V_p}(w_2)=d_{\overline{G}[U_3]}(w_2)-1\le k-2$. Since $|U_2|=n-|U_3|-|U_1|= n-2k-1-t_1$ and $n\ge 8k^2-2k$, by Claim \ref{specST-th2-claim6}(iv), we have  $|U_2|\ge 1$.
As $e(U_2,U_3)=0$, we have $uw_2\notin E(G)$ for any $u\in U_2=V_p\setminus U_1$.
Let $u\in U_2$ and
\begin{align}\label{spec-ST-mian2-1}
G_1=G+uw_2-w_2w'+w'w-wu_1.
\end{align}
Combining this with Claim \ref{specST-th2-claim5}(ii),  we get for any $v\in U_3$,
\begin{align}\label{spec-ST-mian2-1-add}
d_{G_1}(v)=d_G(v)=2k-1.
\end{align}
Moreover, by $\triangle({\overline{G}[U_3]})\le k$ in Claim \ref{specST-th2-claim5}(v)
and $d_{\overline{G}[U_3]}(w_2)\le k-1$, we have
\begin{align}\label{spec-ST-mian2-2}
\triangle({\overline{G_1}[U_3]})\le k.
\end{align}
Thus by Lemma \ref{EDST-thm2-le1}, $G_1$ is $(2k-1)$-edge-connected.
As for any $v\in U_3$,  $d_{\overline{G}[U_3]}(v)=d_{V_p}(v)+1$ in Claim \ref{specST-th2-claim5}(iii), we obtain $d_{\overline{G_1}[U_3]}(v)=d_{G_1[V_p]}(v)+1$.  So by \eqref{spec-ST-mian2-2},  for any $v\in U_3$,
\begin{align}\label{spec-ST-mian2-2-1}
 d_{G_1[V_p]}(v) \le k-1.
\end{align}

Let $\mathbf y$ be the Perron vector of $A(G_1)$.
By $A(G_1)\mathbf{y}=\lambda(G_1)\mathbf{y}$ and
Claim \ref{specST-th2-claim5}(i),  for any $z\in V_p$ and $w\in U_3$, we have
\begin{align}
\lambda(G_1)y_{z}&=\sum_{v\in V_p\setminus \{z\}}y_v+\sum_{v\in N_{G_1}(z)\cap U_3}y_v,
\label{spec-ST-mian2-3}
\\
\lambda(G_1)y_{w}&=\sum_{v\in N_{G_1}(w)\cap V_p}y_v+\sum_{v\in N_{G_1}(w)\cap U_3}y_v. \label{spec-ST-mian2-4}
\end{align}

For short, we  write $\lambda_1$ for $\lambda(G_1)$. Let $y_{z_0}=\min\{y_{v}: v\in V_p\}$ and $y_{w_0}=\max\{y_{v}: v\in U_3\}$.
By (\ref{spec-ST-mian2-1-add}), \eqref{spec-ST-mian2-2-1},    (\ref{spec-ST-mian2-3}), and  (\ref{spec-ST-mian2-4}),   we have
\begin{align*}
\lambda_1y_{z_0}-\lambda_1y_{w_0} &= \sum_{v\in V_p\setminus \{z_0\}}y_v+\sum_{v\in N_{G_1}(z_0)\cap U_3}y_v
-\sum_{v\in N_{G_1}(w_0)\cap V_p}y_v
 -\sum_{v\in N_{G_1}(w_0)\cap U_3}y_v
\\
&\ge \sum_{v\in V_p\setminus  \{z_0\}}y_v - \sum_{v\in N_{G_1}(w_0)\cap V_p}y_v
- \sum_{v\in N_{G_1}(w_0)\cap U_3}y_v
\nonumber \\
&\ge \sum_{v\in V_p\setminus  N_{G_1}(w_0)}y_{z_0} - y_{z_0} -(2k-1)y_{w_0}
\nonumber \\
 &\ge \bigl(|V_p|-(k-1)\bigr)y_{z_0}-y_{z_0}-(2k-1)y_{w_0}
\nonumber \\
 &= (n-3k-1)y_{z_0}-(2k-1)y_{w_0} ,
\end{align*}
which implies that
\begin{align}\label{spec-ST-mian2-5}
(\lambda_1-n+3k+1)y_{z_0}\ge (\lambda_1-2k+1)y_{w_0}.
\end{align}
As $K_{n-2k-1}\subsetneqq G_1 \subsetneqq K_n$, we have
 $n-1>\lambda(G_1)=\lambda_1 > n-2k-2$. This, together with $n\ge 8k^2-2k$ and  (\ref{spec-ST-mian2-5}), yields that
\begin{align}\label{spec-ST-mian2-6}
y_{z_0}\ge \frac{\lambda_1-2k+1}{\lambda_1-n+3k+1}y_{w_0}> \frac{n-2k}{3k}y_{w_0}> y_{w_0}.
\end{align}

Since $G_1=G+uw_2-w_2w'+w'w-wu_1$, where $u\in U_2$,  by Claim \ref{specST-th2-claim6}(i), we have $u,u_1\in N_{G_1[V_p]}(w_2)\setminus  N_{G_1[V_p]}(w)$.
Combining these with \eqref{spec-ST-mian2-2}, (\ref{spec-ST-mian2-4}),   and  Claim \ref{specST-th2-claim6}(ii),   we have
\begin{align*}
&\lambda_1y_{w_2} - \lambda_1y_{w}
\nonumber\\
=&
\sum_{v\in N_{G_1}(w_2)\cap   V_p}y_v
  - \sum_{v\in N_{G_1}(w)\cap V_p}y_v
  +\sum_{v\in N_{G_1}(w_2)\cap U_3}y_v
  - \sum_{v\in N_{G_1}(w)\cap U_3}y_v
  \nonumber\\
=&
\sum_{v\in N_{G_1}(w_2)\cap V_p}y_v
  - \sum_{v\in N_{G_1}(w)\cap  V_p}y_v
  +\big(\sum_{v\in U_3}y_v
  -\sum_{v\in N_{{\overline{G_1}}[U_3]}[w_2]}y_v\big)
   - \big(\sum_{v\in U_3}y_v
  -\sum_{v\in N_{{\overline{G_1}}[U_3]}[w]}y_v\big)
   \nonumber\\
=&
\sum_{v\in N_{G_1}(w_2)\cap V_p}y_v
  - \sum_{v\in N_{G_1}(w)\cap V_p}y_v
    +\sum_{v\in N_{{\overline{G_1}}[U_3]}[w]}y_v
  -\sum_{v\in N_{{\overline{G_1}}[U_3]}[w_2]}y_v
\nonumber \\
\ge &
y_{u}+y_{u_1}+ y_{w} -y_{w_2} -\sum_{v\in N_{{\overline{G_1}}[U_3]}(w_2)}y_v
\nonumber \\
\ge &
y_{u}+y_{u_1}+ y_{w} -y_{w_2} -  k y_{w_0}
\\
\ge & 2y_{z_0} -  k y_{w_0} + y_{w} -y_{w_2}.
\end{align*}
Combining this with (\ref{spec-ST-mian2-6}) and $n\ge 8k^2-2k$, we have
\begin{equation}\label{spec-ST-mian2-7}
y_{w_2} - y_{w}\ge \frac{2y_{z_0} - k y_{w_0}}{\lambda_1+1}\ge
\frac{2n-4k-3k^2}{3k(\lambda_1+1)}y_{w_0} >0.
\end{equation}

As $G_1=G+uw_2-w_2w'+w'w-wu_1$, where $u\in U_2$,  we have $N_{G_1[U_3]}(u_1)=N_{U_3}(u_1)\setminus \{w\}$ and $N_{G_1[U_3]}(u)= \{w_2\}$. Combining these with (\ref{spec-ST-mian2-3})
and  $t_2=d_{U_3}(u_1)$ in Claim \ref{specST-th2-claim6}(ii),   we have
\begin{align*}
\lambda_1y_{u} - \lambda_1y_{u_1}&= y_{u_1} - y_{u} + \sum_{v\in N_{G_1}(u)\cap U_3}y_v  - \sum_{v\in N_{G_1}(u_1)\cap U_3}y_v
\nonumber \\
&= y_{u_1} - y_{u}+y_{w_2} - \sum_{v\in N_{G_1}(u_1)\cap (U_3\setminus\{w\})}y_v
\nonumber \\
&\ge
y_{u_1}-y_{u}-(t_2-1)y_{w_0}.
\end{align*}
Thus
\begin{equation}\label{spec-ST-mian2-8}
y_{u} - y_{u_1}\ge \frac{1- t_2}{\lambda_1+1}y_{w_0}.
\end{equation}

As $e(V_p, U_3)=2k-1$ in Claim \ref{specST-th2-claim5}(ii) and $t_2\le e(V_p, U_3)$, we have $t_2\le 2k-1$. Since $w\in U_3\setminus\{w_1,w_2\}$,  it follows that $N_{V_p}(w) \subseteq N_{V_p}(w_2)$ by Claim \ref{specST-th2-claim6}(ii). Thus by Claim \ref{specST-th2-claim6}(iii), we have  $x_{w_2}\ge x_{w}$.
Combining this with \eqref{spec-ST-mian2-6}, \eqref{spec-ST-mian2-8}, \eqref{specST-th2-eq17}, \eqref{spec-ST-mian2-7}, \eqref{specST-th2-eq66}, $t_2\le 2k-1$, $n-1>\lambda_1 > n-2k-2$, $\lambda > n-2k-2$   and  $n\ge 8k^2-2k$, we have
\begin{align*}
&{\mathbf y}^T (\lambda(G_1)-\lambda(G)) {\mathbf x}= {\mathbf y}^T (A(G_1)-A(G)) {\mathbf x}
\\
=& y_{u}  x_{w_2}+y_{w_2}x_{u} -y_{w_2}x_{w'}-y_{w'}x_{w_2}
+y_{w'}x_{w}+y_{w}x_{w'}-y_{w}x_{u_1}-y_{u_1}x_{w}
\\
=& (y_{u} -y_{w'} )(x_{w_2}-x_{w})+(y_{u} -y_{u_1})x_{w}+
(x_{u} -x_{w'} )(y_{w_2}-y_{w})+(x_{u} -x_{u_1})y_{w}
\nonumber\\
\ge& (y_{u} -y_{u_1})x_{w}+
(x_{u} -x_{w'} )(y_{w_2}-y_{w})+(x_{u} -x_{u_1})y_{w}
\\
\ge& (y_{u} -y_{u_1})x_{w}+
(x_{u_0} -x_{v_0} )(y_{w_2}-y_{w})+(x_{u} -x_{u_1})y_{w}
\nonumber\\
\ge & \frac{1-t_2}{\lambda_1+1}y_{w_0}x_{w} + \frac{(2n-4k-3k^2)(n-6k)}{12k^2(\lambda_1+1)}x_{v_0} y_{w_0}-\frac{t_2}{\lambda+1}x_{v_0}y_{w}
\\
>& \left( -\frac{2t_2}{n-2k-1}  + \frac{(2n-4k-3k^2)(n-6k)}{12k^2n}\right)x_{w_0}y_{w_0}
\nonumber\\
\ge & \left( -\frac{2(2k-1)}{n-2k-1}  + \frac{(2n-4k-3k^2)(n-6k)}{12k^2n}\right)x_{w_0}y_{w_0}
\\
=& \frac{\left[(2n-3k^2-20k-2)n -24k^3+83k^2+16k \right]n -36k^4-66k^3-24k^2}
{12k^2n(n-2k-1)}  x_{w_0}y_{w_0}
>0,
\end{align*}
a contradiction.  Hence, $W_3= \emptyset$. This proves  Claim \ref{specST-th2-claim7-0}.\q

\begin{Claim}\label{specST-th2-claim7}
If $k\ge 3$, then each of the following holds:
\\
(i) $U_3=W_4$ when $d_{\overline{G}[U_3]}(w_2)\le k-1$;
\\
(ii) $\overline{G}[U_3]$ is a tree with $d_{\overline{G}[U_3]}(w_1)=d_{\overline{G}[U_3]}(w_2)=k$.
\end{Claim}

\p  By the definition of $W_4$, we have $W_4\subseteq U_3$.
We next prove $U_3\subseteq W_4$.
On the contrary, suppose that $U_3\nsubseteq W_4$.
By $W_4\subseteq U_3$,  we have $W_4\subsetneqq U_3$.
Choose a vertex $w\in U_3\setminus W_4$ such that $d_{\overline{G}[U_3]}(w)$ is as large as possible.
By Claim \ref{specST-th2-claim7-0}, we have $N_{\overline{G}[U_3]}(w)\cap  (W_1\setminus \{w_1,w_2\})=\emptyset $.
Thus $N_{\overline{G}[U_3]}(w)\subseteq  W_2$.
As $U_3=W$ in Claim \ref{specST-th2-claim5}(iv), we have $d_{\overline{G}[U_3]}(w)\ge 1$. Combining this with  $N_{\overline{G}[U_3]}(w)\subseteq  W_2$  and Claim \ref{specST-th2-claim6-7}(ii), we have  $d_{\overline{G}[U_3]}(w)\ge 2$.
Let $w'\in N_{\overline{G}[U_3]}(w)$. Since $N_{\overline{G}[U_3]}(w)\subseteq  W_2$, we have $d_{\overline{G}[U_3]}(w')=1$.
This implies that  $w'\in W_3$, contrary to $W_3 =\emptyset$ by   Claim \ref{specST-th2-claim7-0}. Thus,  Claim \ref{specST-th2-claim7}(i) holds.

To show  Claim \ref{specST-th2-claim7}(ii), it suffices to show that    $d_{\overline{G}[U_3]}(w_2)= k$. Firstly, if $d_{\overline{G}[U_3]}(w_2)= k$, then   Claim \ref{specST-th2-claim5}(iii) yields that  $d_{\overline{G}[U_3]}(w_1)=d_{\overline{G}[U_3]}(w_2)=k$.
If  $\overline{G}[W_4]$ is disconnected, then   $|U_3|\ge d_{\overline{G}[U_3]}(w_1)+1+d_{\overline{G}[U_3]}(w_2)+1=2k+2$, contrary to $|U_3|=2k+1$. Thus we have $\overline{G}[W_4]$ is connected.
If  $\overline{G}[U_3]$ is disconnected, then by Claim \ref{specST-th2-claim6-7}(ii), we have  $e(\overline{G}[U_3]) \ge e(\overline{G}[W_4])+2\ge 2k-1+2=2k+1$, contrary to $e(\overline{G}[U_3])=2k$ in Claim \ref{specST-th2-claim5}(ii). Thus we have $\overline{G}[U_3]$ is connected.
Combining this with $|U_3|=2k+1=e(\overline{G}[U_3])+1$ in Claim \ref{specST-th2-claim5}(iv), we get that
$\overline{G}[U_3]$ is a tree.

Next, we will show that $d_{\overline{G}[U_3]}(w_2)= k$ by contradiction. Suppose that $d_{\overline{G}[U_3]}(w_2)\neq k$. By Claim \ref{specST-th2-claim5}(v), we have
$d_{\overline{G}[U_3]}(w_2)\le k-1$. Combining this with Claims \ref{specST-th2-claim7}(i) and  \ref{specST-th2-claim5}(v), we have $$|U_3|\le d_{\overline{G}[U_3]}(w_1)+1+d_{\overline{G}[U_3]}(w_2)+1\le 2k+1.$$
As $|U_3|=2k+1$, we have $d_{\overline{G}[U_3]}(w_1)=k$,  $d_{\overline{G}[U_3]}(w_2)=k-1$  and $e_{\overline{G}}(N_{\overline{G}[U_3]}[w_1], N_{\overline{G}[U_3]}[w_2])=0$.
Thus by Claim \ref{specST-th2-claim7-0}, we have $\overline{G} (N_{\overline{G}[U_3]}[w_1])=K_{1,k}$. Combining this with $e(\overline{G}[U_3])=2k$ in Claim \ref{specST-th2-claim5}(ii), we get
$e({\overline{G}}[N_{\overline{G}[U_3]}[w_2]]=e(\overline{G}[U_3])-
e({\overline{G}}[N_{\overline{G}[U_3]}[w_1]]=2k-k=k$. Then $\overline{G} (N_{\overline{G}[U_3]}[w_2])=K_1\vee (K_2+(k-3)K_1)$. Therefore, $\overline{G}[U_3]=K_{1,k}+K_1\vee (K_2+(k-3)K_1)$.
By Claims  \ref{specST-th2-claim5}(iii) and    \ref{specST-th2-claim6}(i)-(iii), we have $w_3,w_4\in N_{\overline{G}[U_3]}(w_2)$, $\overline{G}[\{w_2,w_3,w_4\}]=K_3$,
$u_1w_3\in E(G)$, $u_1w_4\in E(G)$, $N_{V_p}(w_1)=\{u_1,u_2,\ldots,u_{k-1}\}$ and $N_{V_p}(w_2)=\{u_1,u_2,\ldots,u_{k-2}\}$.
Let $w\in N_{\overline{G}[U_3]}(w_1)$. As $\overline{G} (N_{\overline{G}[U_3]}[w_1])=K_{1,k}$, we have $N_{\overline{G}[U_3]}(w)=\{w_1\}$.  Thus  $ww_2\in E(G)$ and  $u_1w\notin E(G)$ by Claim \ref{specST-th2-claim5}(iii).
Now, $G$ is a graph as shown in   Figure \ref{fig3}.
By \eqref{specST-th2-eq15},
we have  $x_{w_3}=x_{w_4}$,
\begin{align*}
\lambda x_{w_2}=&\sum_{v\in N_G(w_2)\cap U_3 }x_w+\sum_{v\in N_G(w_2)\cap V_p} x_v >  \sum_{v\in N_G(w_2)\cap V_p} x_v \ge (k-2)x_{u_0},\,\,\text{and}
\\
\lambda x_{w_3} =&\sum_{v\in N_G(w_3)\cap U_3 }x_w+\sum_{v\in N_G(w_3)\cap V_p} x_v \le x_{u_1}+\sum_{v\in N_G(w_3)\cap U_3} x_v \le
x_{u_1}+(2k-2)x_{v_0},
\end{align*}
as  $N_{\overline{G}[U_3]}(w_3)=\{w_2,w_4\}.$ Thus
\begin{align}\label{specST-th2-eq91}
 x_{w_2} > \frac{k-2}{\lambda} x_{u_0}  \text{~and~}
 x_{w_3}\le \frac{1}{\lambda} x_{u_1}+\frac{2k-2}{\lambda} x_{v_0},
 \end{align}
Since  $\{w_1,w_2,w_3,w_4\}\subseteq N_{U_3}(u_1)$ and $u_0\in U_2$,  it follows from \eqref{specST-th2-eq13} that
\begin{align*}
  \lambda x_{u_1}-\lambda x_{u_0}=  \sum_{v\in V_p\setminus \{u_1\}}x_v -\sum_{v\in V_p\setminus \{u_0\}}x_v+ \sum_{v\in N_G(u_1)\cap U_3 }x_v-\sum_{v\in N_G(u_0)\cap U_3 }x_v
\le   x_{u_0} - x_{u_1}  + 4x_{v_0},
\end{align*}
which implies that
\begin{align}\label{specST-th2-eq91-1}
 x_{u_1}\le  x_{u_0}+\frac{4}{\lambda+1}x_{v_0}.
 \end{align}

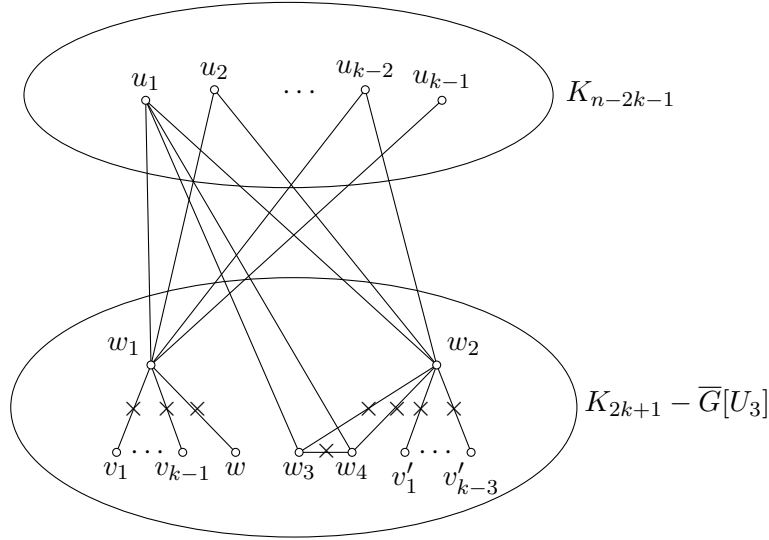
\begin{figure}[htbp]
\centering
\begin{tikzpicture}[scale=0.7]

\coordinate (u1) at (-2.8,5.8);
\coordinate (u2) at (-1.5,6.0);
\coordinate (u{k-1}) at (1.35,6.0);
\coordinate (uk) at ( 2.8,5.8);

\coordinate (w1) at (-2.7,0.8);
\coordinate (w2) at ( 2.7,0.8);

\coordinate (v1)  at (-3.35,-0.85);
\coordinate (vk)  at (-2.10,-0.85);
\coordinate (v{k+1}) at (-1.10,-0.85);
\coordinate (w3) at (0.10,-0.85);
\coordinate (w4) at (1.10,-0.85);
\coordinate (vp1) at ( 2.10,-0.85);
\coordinate (vpk) at ( 3.35,-0.85);
\draw (-0.1,5.9) ellipse (5.0cm and 1.75cm);
\draw ( 0.0,0.0) ellipse (5.35cm and 2.45cm);
\node[right] at (4.95,5.90) {$K_{n-2k-1}$};
\node[right] at (5.30,0.10) {$K_{2k+1}-\overline{G}[U_3]$};

\draw (u1)--(w1);
\draw (u1)--(w2);
\draw (u1)--(w3);
\draw (u1)--(w4);
\draw (u2)--(w1);
\draw (u2)--(w2);
\draw (uk)--(w1);
\draw (u{k-1})--(w1);
\draw (u{k-1})--(w2);

\draw (w3)--(w2);
\draw (w1)--(v1);
\draw (w1)--(vk);
\draw (w1)--(v{k+1});
\draw (w2)--(vp1);
\draw (w2)--(vpk);
\draw (w2)--(w4);
\draw (w3)--(w4);
\node[inner sep=2pt] at (-3.03,-0.02) {$\times$};
\node[inner sep=2pt] at (-2.40,-0.02) {$\times$};
\node[inner sep=2pt] at (-1.82,-0.02) {$\times$};
\node[inner sep=2pt] at ( 1.42,-0.02) {$\times$};
\node[inner sep=2pt] at (1.95,-0.02) {$\times$};
\node[inner sep=2pt] at ( 2.40,-0.02) {$\times$};
\node[inner sep=2pt] at ( 3.03,-0.02) {$\times$};
\node[inner sep=2pt] at ( 0.62,-0.82) {$\times$};

\draw[fill=white] (u1) circle (0.075cm);
\draw[fill=white] (u2) circle (0.075cm);
\draw[fill=white] (u{k-1}) circle (0.075cm);
\draw[fill=white] (uk) circle (0.075cm);
\draw[fill=white] (w1) circle (0.075cm);
\draw[fill=white] (w3) circle (0.075cm);
\draw[fill=white] (w4) circle (0.075cm);
\draw[fill=white] (w2) circle (0.075cm);
\draw[fill=white] (v1) circle (0.075cm);
\draw[fill=white] (vk) circle (0.075cm);
\draw[fill=white] (v{k+1}) circle (0.075cm);
\draw[fill=white] (vp1) circle (0.075cm);
\draw[fill=white] (vpk) circle (0.075cm);

\node[above] at (u1) {$u_1$};
\node[above] at (u2) {$u_2$};
\node[above] at (u{k-1}) {$u_{k-2}$};
\node[above] at (uk) {$u_{k-1}$};
\node at (0.15,5.95) {$\cdots$};

\node[above left]  at (w1) {$w_1$};
\node[above right] at (w2) {$w_2$};

\node[below]  at (w3) {$w_3$};
\node[below]  at (w4) {$w_4$};

\node[below] at (v1)  {$v_1$};
\node[below] at (vk)  {$v_{k-1}$};
\node[below] at (v{k+1})  {$w$};
\node[below] at (vp1) {$v'_1$};
\node[below] at (vpk) {$v'_{k-3}$};
\node at (-2.72,-0.85) {$\cdots$};
\node at ( 2.72,-0.85) {$\cdots$};

\end{tikzpicture}
\caption{The graph $G$.~~~~~~~~}\label{fig3}
\end{figure}

Let $G_2=G+u_1w-ww_2+w_2u_{k-1}-u_1w_3+w_3w_4-w_4u_1$. Then $G_2\cong H_2(k)$.
Let $\boldsymbol{\eta}$ be the Perron vector of $A(G_2)$ and  $\eta_v$ be the coordinate in $\boldsymbol{\eta}$  corresponding to  vertex $v\in V(G_2)$.
For short, write $\lambda_2$ for $\lambda(G_2)$.
Let $\eta_{z_1}=\min\{\eta_{w}: w\in V_p\}$ and $ \eta_{z_2}=\max\{x_{w}: w\in U_3\}$. Since $e(U_3,V_p)=2k-1<n-2k-1=|V_p|$, we have $N_G[z_1]=V_p$ by the choice of $z_1$. Thus, $z_1z_2\not\in E(G_2)$.  As $N_{G_2}(w_3)\setminus\{w_4\}=N_{G_2}(w_4)\setminus\{w_3\}$, we get  $\eta_{w_3}=\eta_{w_4}$.  Furthermore,  we also have $d_{G_2[V_p]}(w_2)=d_{G[V_p]}(w_2)+1=k-1$,
$d_{G_2}(w_3)=d_{G}(w_3)=2k-1$ and $d_{G_2[V_p]}(w_3)=d_{G}(w_3)-1=0$. Thus
 \begin{align*}
 \lambda_2 \eta_{w_2}=& \sum_{v\in N_{G_2}(w_2)\cap V_p} \eta_v+\sum_{v\in N_{G_2}(w_2)\cap U_3} \eta_v >\sum_{v\in N_{G_2}(w_2)\cap V_p} \eta_v \ge (k-1)\eta_{z_1},
 \\
 \lambda_2 \eta_{w_3}=& \sum_{v\in N_{G_2}(w_3)\cap U_3} \eta_v \le   (2k-1)\eta_{z_2},
 \end{align*}
 which implies that
 \begin{align}\label{specST-th2-eq92}
 \eta_{w_2} \ge \frac{k-1}{\lambda_2}\eta_{z_1} \text{~and~}
 \eta_{w_3} \le \frac{2k-1}{\lambda_2}\eta_{z_2}.
 \end{align}
Since  $N_{G_2[U_3]}(u_1) \supseteq  N_{G_2[U_3]}(u)$ for any $u\in V_p$  and $d_{G_2[U_3]}(u_1)=3$,  we get for any $u\in V_p$,
\begin{align*}
 \lambda_2 \eta_{u_1} - \lambda_2 \eta_{u} = & \sum_{v\in N_{G_2}(u_1)\cap V_p}\eta_v-\sum_{v\in N_{G_2}(u)\cap V_p}\eta_v + \sum_{v\in N_{G_2}(u_1)\cap U_3}\eta_v - \sum_{v\in N_{G_2}(u)\cap U_3}\eta_v
\\
=& \eta_{u} - \eta_{u_1} + \sum_{v\in N_{G_2}(u_1)\cap U_3}\eta_v -\sum_{v\in N_{G_2}(u)\cap U_3}\eta_v
\ge   \eta_{u} - \eta_{u_1},
\end{align*}
and
\begin{align*}
 \lambda_2 \eta_{u} - \lambda_2 \eta_{u_1} = & \sum_{v\in N_{G_2}(u)\cap V_p}\eta_v- \sum_{v\in N_{G_2}(u_1)\cap V_p}\eta_v +\sum_{v\in N_{G_2}(u)\cap U_3}\eta_v - \sum_{v\in N_{G_2}(u_1)\cap U_3}\eta_v
\\
=& \eta_{u_1} - \eta_{u} +\sum_{v\in N_{G_2}(u)\cap U_3}\eta_v- \sum_{v\in N_{G_2}(u_1)\cap U_3}\eta_v
\\
\ge & \eta_{u_1} - \eta_{u} - \sum_{v\in N_{G_2}(u_1)\cap U_3}\eta_v \ge \eta_{u_1} - \eta_{u} - 3\eta_{z_2},
\end{align*}
which implies that
\begin{align}\label{specST-th2-eq93}
  \eta_{u} \le \eta_{u_1}\le \eta_{u} + \frac{ 3}{\lambda_2+1}\eta_{z_2}.
 \end{align}
Since $z_1z_2\not\in E(G_2)$, $N_{G_2[V_p]}[u]=N_{G[V_p]}[u]=V_p$ for any $u\in V_p$ and $d_{G_2}(v)=d_{G}(v)=2k-1$ for any $v\in U_3$, it follows that
 \begin{align*}
 \lambda_2 \eta_{z_1} - \lambda_2 \eta_{z_2}
 =& \sum_{v\in N_{G_2}(z_1)\cap V_p} \eta_v+\sum_{v\in N_{G_2}(z_1)\cap U_3} \eta_v - \sum_{v\in N_{G_2}(z_2)\cap V_p} \eta_v-\sum_{v\in N_{G_2}(z_2)\cap U_3} \eta_v
 \\
 \ge & \sum_{v\in N_{G_2}(z_1)\cap V_p} \eta_v - \sum_{v\in N_{G_2}(z_2)\cap V_p} \eta_v-\sum_{v\in N_{G_2}(z_2)\cap U_3} \eta_v
  \\
 \ge & \bigl(d_{G_2[V_p]}(z_1)-d_{G_2[V_p]}(z_2)\bigr) \eta_{z_1}
     -\bigl(d_{G_2}(z_2)-d_{G_2[V_p]}(z_2)\bigr)\eta_{z_2}
 \\
 \ge & \bigl(n-2k-2-d_{G_2[V_p]}(z_2)\bigr) \eta_{z_1}
     -\bigl(2k-1-d_{G_2[V_p]}(z_2)\bigr)\eta_{z_2} ,
 \end{align*}
 which implies that
\begin{align}\label{specST-th2-eq100}
\eta_{z_1} \ge \frac{\lambda_2-2k+1+d_{G_2[V_p]}(z_2)}{\lambda_2-n+2k+2+d_{G_2[V_p]}(z_2) }\eta_{z_2}.
 \end{align}
As $K_{n-2k-1}\subsetneqq G_2 $, we have $n-1>\lambda_2>n-2k-2$.  By
\eqref{specST-th2-eq100} and $n\ge 8k^2-2k$, we have $\eta_{z_1}>\eta_{z_2}$.  As $G_2\cong H_2(k)$, we have for any $v\in U_3$, $0\le d_{G_2[V_p]}(v)\le k-1$. Thus by \eqref{specST-th2-eq93},
 \begin{align*}
\lambda_2 \eta_{z_2}=\sum_{v\in N_{G_2}(z_2)\cap V_p} \eta_v+\sum_{v\in N_{G_2}(z_2)\cap U_3} \eta_v \le (k-1) \eta_{u_1}+k\eta_{z_2} \le (k-1) (\eta_{z_1}+\eta_{z_2})+k\eta_{z_2}.
 \end{align*}
Combining this with $\lambda_2>n-2k-2$, $n\ge 8k^2-2k$ and $k\ge 3$, we have
\begin{align}\label{specST-th2-eq101}
\eta_{z_1} \ge \frac{\lambda_2-2k+1}{k-1}\eta_{z_2}> \frac{n-4k-1}{k-1}\eta_{z_2}\ge \frac{8k^2-6k-1}{k-1}\eta_{z_2}\ge 8k\eta_{z_2}.
\end{align}
As $d_{G_2[U_3]}(u_1)=3$, we have
\begin{align*}
\lambda_2 \eta_{u_1} =     \sum_{v\in N_{G_2}(u_1)\cap V_p}\eta_v +  \sum_{v\in N_{G_2}(u_1)\cap U_3}\eta_v
\le   (n-2k-2)\eta_{u_1} + 3\eta_{z_2}< (n-2k+1)\eta_{u_1}.
\end{align*}
Thus $\lambda_2 < n-2k+1$.

By \eqref{specST-th2-eq91-1}, Claim \ref{specST-th2-claim5}(ii)--(iii) and (v),  we have
\begin{align*}
\lambda x_{v_0}=\sum_{v\in N_{G}(v_0)\cap V_p} x_v+\sum_{v\in N_{G}(v_0)\cap U_3} \eta_v \le (k-1) x_{u_1}+k x_{v_0} \le (k-1) (x_{u_0}+x_{v_0})+k x_{v_0}.
 \end{align*}
Combining this with $\lambda>n-2k-2$, $n\ge 8k^2-2k$ and $k\ge 3$, we have
\begin{align}\label{specST-th2-eq101-1}
x_{u_0} \ge \frac{\lambda-2k+1}{k-1}x_{v_0}> \frac{n-4k-1}{k-1}x_{v_0}\ge \frac{8k^2-6k-1}{k-1}x_{v_0}\ge 8kx_{v_0}.
\end{align}
As $d_{U_3}(u_1)=4$, we have
\begin{align*}
\lambda x_{u_1} =     \sum_{v\in N_{G}(u_1)\cap V_p}x_v +  \sum_{v\in N_{G}(u_1)\cap U_3}x_v
\le   (n-2k-2)x_{u_1} + 4x_{z_2}< (n-2k+2)x_{u_1}.
\end{align*}
Thus $\lambda < n-2k+2$.

Let $r=n-2k-2$ and $R=n-2k+2$.  Then $r<\lambda<R$ and $r<\lambda_2<R$. Combining these with $x_{w_3}=x_{w_4}$, $\eta_{w_3}=\eta_{w_4}$, \eqref{specST-th2-eq91}, \eqref{specST-th2-eq91-1}, \eqref{specST-th2-eq92},  \eqref{specST-th2-eq93}, \eqref{specST-th2-eq101}, \eqref{specST-th2-eq101-1}  and $n\ge 8k^2-2k$,  we have
 \begin{align*}
&{\mathbf \eta}^T (\lambda(G_2)-\lambda(G)) {\mathbf x}= {\mathbf \eta}^T (A(G_2)-A(G)) {\mathbf x}
\\
= &  \eta_{u_1}x_w+\eta_wx_{u_1}-\eta_wx_{w_2}-\eta_{w_2}x_w
     +\eta_{w_2}x_{u_{k-1}}+\eta_{u_{k-1}}x_{w_2}-\eta_{u_1}x_{w_3}
\\
  & -\eta_{w_3}x_{u_1}+\eta_{w_3}x_{w_4}+\eta_{w_4}x_{w_3}
    -\eta_{w_4}x_{u_1}-\eta_{u_1}x_{w_4}
\\
= &  \eta_{w_2}(x_{u_{k-1}}-x_w) +x_{w_2}\eta_{u_{k-1}}
   +\eta_w(x_{u_1}-x_{w_2}) +\eta_{u_1}x_w
   -2\eta_{u_1}x_{w_3}+2\eta_{w_3}(x_{w_3}-x_{u_1})
\\
> &  \eta_{w_2}(x_{u_0}-x_{v_0}) +x_{w_2}\eta_{z_1}
     -2\eta_{u_1}x_{w_3}-2\eta_{w_3}x_{u_1}
\\
\ge & \frac{k-1}{\lambda_2}\eta_{z_1}(x_{u_0}-x_{v_0})
   + \frac{k-2}{\lambda}x_{u_0}\eta_{z_1}
   - 2\big(\eta_{z_1} + \frac{3}{\lambda_2+1}\eta_{z_2}\big)
       \big(\frac{1}{\lambda}x_{u_1} +\frac{2k-2}{\lambda} x_{v_0}\big)
 \\
   & - \frac{4k-2}{\lambda_2}\eta_{z_2}\big(x_{u_0}+\frac{4}{\lambda+1}x_{v_0}\big)
 \\
 = & \frac{k-1}{\lambda_2}\eta_{z_1} x_{u_0}-
     \frac{k-1}{\lambda_2}\eta_{z_1} x_{v_0}
   + \frac{k-2}{\lambda}x_{u_0}\eta_{z_1}
   - \frac{2}{\lambda}\eta_{z_1}x_{u_1} - \frac{4k-4}{\lambda}\eta_{z_1}x_{v_0}
  \\
   & -\frac{6}{(\lambda_2+1)\lambda}\eta_{z_2}x_{u_1}
  - \frac{4k-2}{\lambda_2}\eta_{z_2} x_{u_0}  - \frac{12k-12}{\lambda(\lambda_2+1)}\eta_{z_2} x_{v_0} - \frac{16k-8}{\lambda_2(\lambda+1)}\eta_{z_2} x_{v_0}
\\
\ge & \frac{k-1}{\lambda_2}\eta_{z_1} x_{u_0}- \frac{k-1}{\lambda_2}\eta_{z_1} x_{v_0}
   + \frac{k-2}{\lambda}x_{u_0}\eta_{z_1}
   - \frac{2}{\lambda}\eta_{z_1}\big(x_{u_0}+\frac{4}{\lambda+1}x_{v_0}\big)  - \frac{4k-4}{\lambda}\eta_{z_1}x_{v_0}
  \\
  & -\frac{6}{(\lambda_2+1)\lambda}\eta_{z_2}\big(x_{u_0}+\frac{4}{\lambda+1}x_{v_0}\big)
  - \frac{4k-2}{\lambda_2}\eta_{z_2} x_{u_0} - \frac{28k-20}{r^2}\eta_{z_2} x_{v_0}
\\
= & \bigl(\frac{k-1}{\lambda_2} + \frac{k-4}{\lambda}\bigr) x_{u_0}\eta_{z_1}
- \frac{k-1}{\lambda_2}\eta_{z_1} x_{v_0}
- \frac{4k-4}{\lambda}\eta_{z_1}x_{v_0}
   - \frac{8\eta_{z_1}x_{v_0}}{\lambda(\lambda+1)}
\\
  & -\frac{6}{(\lambda_2+1)\lambda}\eta_{z_2}x_{u_0}
  -\frac{24}{(\lambda_2+1)\lambda(\lambda+1)}\eta_{z_2}x_{v_0}
  - \frac{4k-2}{\lambda_2}\eta_{z_2} x_{u_0} - \frac{28k-20}{r^2}\eta_{z_2} x_{v_0}
  \\
\ge & \bigl(\frac{2k-5}{R}- \frac{k-1}{8kr}- \frac{4k-4}{8kr}
   - \frac{1}{kr^2} -\frac{3}{4kr^2 }-\frac{3}{8k^2 r^3}
  - \frac{2k-1}{4kr}  - \frac{7k-5}{16k^2r^2}  \bigr) x_{u_0}\eta_{z_1}
  \\
= & \bigl(\frac{2k-5}{R}- \frac{9k-7}{8kr} - \frac{35kr-5r+6}{16k^2r^3}
     \bigr) x_{u_0}\eta_{z_1}
\\
\ge  & \bigl(\frac{2k-5}{R}- \frac{9k-6}{8kr}
     \bigr) x_{u_0}\eta_{z_1}
\\
=  &  \frac{k\bigl[(16k-40)r-9R\bigr]+6R}{8krR}
     x_{u_0}\eta_{z_1}>0.
\end{align*}
Thus we have $\lambda(G_2)>\lambda(G)$, a contradiction. So Claim  \ref{specST-th2-claim7}(ii) holds. This proves Claim  \ref{specST-th2-claim7}.
 \q

\par\medskip
\noindent{\bf Proof of Theorem \ref{spectral-spanning-tree-thm2}:}
By Claim \ref{specST-th2-claim6v}, we may suppose that $k\ge 3$. By Claim  \ref{specST-th2-claim7}(ii), $\overline{G}[U_3]$ is a tree with $d_{\overline{G}[U_3]}(w_1)=d_{\overline{G}[U_3]}(w_2)=k$. If $w_1w_2\in E(G)$, then \(w_1\) and \(w_2\) have exactly one common neighbor, as   \(\overline{G}[U_3]\) is a tree with $2k$ edges.
Consequently,
$|N_{\overline{G}[U_3]}(w_1)\cap N_{\overline{G}[U_3]}(w_2)|=1$.
It follows that \(\overline{G}[U_3]\) is obtained from two copies
of \(K_{1,k}\) by identifying one leaf of one copy with one leaf
of the other. By Claims \ref{specST-th2-claim5}(iii) and \ref{specST-th2-claim6}(ii), we have $G_2\cong H_2(k)$, and thus the result already holds.

Next, we assume that  $w_1w_2\in E(\overline{G})$.
Then $e(\overline{G}[W_4])= d_{\overline{G}[U_3]}(w_1)+d_{\overline{G}[U_3]}(w_2)-1=2k-1$ and
$\overline{G}[W_4]$  can be obtained from $K_2$ by attaching $k-1$ pendant vertices to each endpoint.
By Claim \ref{specST-th2-claim5}(iv), we get $|U_3\setminus W_4|=1$ and $d_{\overline{G}[U_3]}(w)=1$, where $\{w\}=U_3\setminus W_4$.
Hence, $\overline{G}[U_3]$ is a tree with  $d_{\overline{G}[U_3]}(w_1)=d_{\overline{G}[U_3]}(w_2)=k\ge 3$, $d_{\overline{G}[U_3]}(w_3)=2$ and $ww_3\in E(\overline{G})$.

As $w_3\in W_4$, $d_{\overline{G}[U_3]}(w_3)=2$  and $ww_3\in E(\overline{G})$,
we get $w_3$ is adjacent to a unique vertex of $\{w_1,w_2\}$. Without loss of generality, suppose that   $w_1w_3\in E(\overline{G})$ and $w_2w_3\in E(G)$.
Let $G_3=G+w_1w_2 -w_2w_3 +ww_3 -w_1w$. By Claims \ref{specST-th2-claim5}(iii) and \ref{specST-th2-claim6}(ii), we have $G_3\cong H_2(k)$ and $\emptyset=N_{V_p}(w)\subsetneqq N_{V_p}(w_3)\subsetneqq N_{V_p}(w_2)= N_{V_p}(w_1)$. Furthermore, by Claim \ref{specST-th2-claim6}(iii), we have
$x_{w_j}>x_{w_3}$ and $x_{w_j}>x_w$  for any $j\in \{1,2\}$.
Thus
\begin{align*}
	\mathbf{x}^{T}(A(G_3)-A(G))\mathbf{x}
	&=2(x_{w_1}x_{w_2}-x_{w_2}x_{w_3}+x_wx_{w_3}-x_{w_1}x_w)
	\\
	&=2(x_{w_1}-x_{w_3})(x_{w_2}-x_w)>0.
\end{align*}
Therefore, $\lambda(G_3)>\lambda(G)$, a contradiction. This completes the proof of Theorem \ref{spectral-spanning-tree-thm2}. \q


\section*{Acknowledgements}
Jia Wei was supported by the Tianyuan Fund for Mathematics of the National Natural Science Foundation of China (No.~12526607), 
Muhuo Liu was supported by the Natural Science Foundation of Guangdong Province (No.~2024A1515011899) and the National Natural Science Foundation of China (No.~12271182),  Zhifu You was supported by the scientific research capacity improvement project of the doctoral program construction unit of Guangdong Polytechnic Normal University (NO.~22GPNUZDJS29) and Hong-Jian Lai is supported in part by the  National Natural Science Foundation of China (No.~12471333).


\end{document}